\documentclass[10pt]{article}
\usepackage{lipsum}
\usepackage{amsfonts}
\usepackage{graphicx}
\usepackage{epstopdf}
\usepackage{algorithmic}
\usepackage{amsmath} 
\usepackage{amsthm}
\usepackage{amsfonts,mathrsfs}
\usepackage{amssymb}
\usepackage{color}
\usepackage{graphicx}
\usepackage{epsfig}
\usepackage{subfigure}
\usepackage{enumerate}
\usepackage{algorithm}
\usepackage[hidelinks]{hyperref}
\usepackage{empheq}
\usepackage{bm}
\usepackage{accents}
\usepackage{cite}
\usepackage{balance}
\usepackage{marginnote}

\ifpdf
\DeclareGraphicsExtensions{.eps,.pdf,.png,.jpg}
\else
\DeclareGraphicsExtensions{.eps}
\fi
\newtheorem{definition}{Definition}{\it}{}
\newtheorem{example}{Example}{\it}{}
\newtheorem{corollary}{Corollary}{\it}{}
\newtheorem{proposition}{Proposition}{\it}{}
\newtheorem{lemma}{Lemma}{\it}{}
\newtheorem{theorem}{Theorem}{\it}{}
\newtheorem{remark}{Remark}{\it}{}
\newtheorem{assumption}{Assumption}{\it}{}

\newcommand{\mc}{\mathcal}
\newcommand{\bb}{\mathbb}

\DeclareMathAlphabet{\mathbbmsl}{U}{bbm}{m}{sl}
\newcommand{\bbi}{\mathbbmsl}

\newcommand{\dom}{\operatorname{dom}}

\newcommand{\red}{\textcolor{red}}

\newcommand{\argmin}{\operatorname{argmin}}

\newcommand{\proj}{\mathrm{proj}}

\newcommand{\Id}{\mathrm{Id}}

\newcommand{\diag}{\operatorname{diag}}

\newcommand{\col}{\operatorname{col}}
\newcommand{\zer}{\operatorname{zer}}

\newcommand{\dist}{\mathrm{dist}}
\newcommand{\nc}{\mathrm{N}}
\newcommand{\0}{\mathbf{0}}
\newcommand{\1}{\mathbf{1}}

\newcommand{\gph}{\operatorname{gph}}

\newcommand{\Rmnum}[1]{\expandafter\@slowromancap\romannumeral #1@}
\newcommand{\eod}{\ensuremath{\hfill\Box}}
\newcommand{\qedd}{\ensuremath{\hfill \blacksquare}}

\title{ Bregman algorithms for mixed-strategy generalized Nash equilibrium seeking in a class of mixed-integer games}
\author{Wicak Ananduta and Sergio Grammatico
	\thanks{This work was partially funded by the ERC under research project COSMOS (802348).}
	\thanks{W. Ananduta and S. Grammatico are with the Delft Center of Systems and Control (DCSC), TU Delft, the Netherlands. E-mail addresses: \texttt{\{w.ananduta, s.grammatico\}@tudelft.nl}. }
}

\begin{document}

\maketitle

\begin{abstract}
We consider the problem of computing a mixed-strategy generalized Nash equilibrium (MS-GNE) for a class of games where each agent has both continuous and integer decision variables. Specifically, we propose a novel Bregman forward-reflected-backward splitting and design distributed algorithms that  exploit the problem structure. Technically, we prove convergence to a variational MS-GNE under mere  monotonicity and Lipschitz continuity assumptions, which are typical of continuous GNE problems. Finally, we show the performance of our algorithms via numerical experiments. 
\end{abstract}




\section{Introduction}
Plenty of decision-making problems involve both continuous and integer decision variables. For instance, in power systems, binary variables are used to model the states of devices, e.g. whether they are on or off. Optimizing over these variables along with continuous ones, e.g. the amount of power generated by a generator or consumed by a load, is the essence of several problems in this domain, such as unit commitment \cite{carrion06,bertsimas13} and demand-side management \cite{kim13}. Another example is the problem of charging electric vehicles \cite{vujanic16,cenedese19} where binary and continuous variables represent the modes of electric vehicles (connected/disconnected and charging/discharging) and the power delivered to/from the electric vehicles, respectively. Mixed-integer problems are also present in transportation systems; for example, the coordination problem for automated vehicles  in \cite{fabiani20} uses integer variables to represent line selections whereas the longitudinal accelerations of the vehicles are continuous variables. Furthermore, in economics, some game-theoretic market models are mixed-integer as well  \cite{sagratella17,sagratella20,sagratella19}. 
{These multi-agent decision problems can be classified as generalized mixed-integer (GMI) games, where each agent must solve its own optimization problem with cost functions and constraints that are coupled with the decision variables of other agents. In this context, the usual goal is to compute a generalized Nash equilibrium (GNE), which describes a collective decision such that each agent has no incentive to unilaterally deviate.}

Currently, the study of GMI games and methods to compute their GNE is limited. The works in  \cite{cenedese19,fabiani20,sagratella17,sagratella20} consider potential GMI games, where there exists a global function that tracks the changes of the cost when some players adjust their decisions. Existence of a GNE in potential GMI games is proven in  \cite{sagratella17}. For this class,  \cite{fabiani20,cenedese19,sagratella17} propose iterative GNE-seeking algorithms based on (Gauss-Seidel or Gauss-Southwell type of) best-response methods,   also suitable for solving 2-group partitionable mixed-integer NE problems \cite{sagratella16,sagratella17b}. A more general class of GMI games is considered in \cite{sagratella19}, in which a potential function may however not exist. In fact, \cite{sagratella19} considers that the cost function of each agent is continuous and convex with respect to its own decision variables that are linearly coupled (only) with the integer variables of other agents. 
Moreover, \cite{sagratella19} shows that {the set of GNE of this class of GMI games is finite (and can be empty), and proposes exact equilibrium seeking branch-and-bound algorithms.} 

One of the main challenges in computing a GNE of a GMI game is that the feasible sets for the integer variables are discrete, hence non-convex. One possible approach  to work around finite action spaces in (non-generalized) games, {particularly those that are played repeatedly,} is considering the mixed-strategy (MS) extension, i.e., where the (integer) decisions are chosen randomly based on some probability distributions (mixed strategies). 
The existence of a mixed-strategy equilibrium for finite games was shown by Nash in his seminal work \cite{nash51} and methods to compute such equilibria are still actively researched, e.g.   \cite{arslan07,gao20,mertikopoulos19}. The work in \cite{arslan07} considers an assignment game, which is a potential game, and proposes a regret-based learning and a spatial adaptive play. Interestingly, for the game in \cite{arslan07},  a pure NE can be computed through the mixed-strategy reformulation. 
Furthermore, \cite{gao20} considers monotone finite games and shows that a continuous-time exponentially discounted reinforcement learning is an output strictly equilibrium independent passive (OS-EIP) system that converges to a discrete distribution that is an approximation of an MS Nash equilibrium (MS-NE).  
Such an approximate Nash distribution is obtained via the use of the \textit{soft-max} function, which has a smoothness property required for the OS-EIP framework.  On the other hand, for general finite games,  the dual averaging \cite{mertikopoulos19}  generates a sequence converging to a strict pure NE if  there exists one and if the sequence is initialized close enough. In addition, particularly for two-player zero-sum games,  the ergodic average of the generated sequence converges to the set of MS-NE. 

To the best of our knowledge, the mixed-strategy extension of a GMI game and solutions in the form of mixed-strategy GNE (MS-GNE) have not been studied in the literature.  
As a consequence of considering the mixed-strategy extension, instead of dealing with discrete constraints,  we have a mixed-strategy GNE problem (MS-GNEP), which has only continuous decision variables, namely mixed strategies (discrete probability distributions) and general continuous variables. 

As extensively discussed in \cite{belgioioso17,yi19,belgioioso20,franci20,franci20b}, fixed-point algorithms based on operator-splitting methods are effective to solve GNEPs with continuous variables since they can deal with the least restrictive assumptions for global convergence results; {moreover,} they can be implemented in a distributed manner, {enhancing} scalability and minimum information sharing. 
Nevertheless, none of the existing fixed-point algorithms for continuous GNEPs take the advantage of the structure of the constraints. Such algorithms typically require each agent to perform a Euclidean projection onto its local feasible set or a proximal step, which is a convex program. On the other hand, we observe from mirror descent algorithms for convex and online optimization \cite{beck03,belmega18,gao20b} { that there exist closed-form non-Euclidean projection mappings for certain types of feasible sets.} 
{The works in \cite{gao20,mertikopoulos19,gao20b} for non-generalized games also adopt the mirror mapping approach albeit do not consider fixed-point methods.}  
{Essentially, a mirror mapping is a Bregman projection, which uses the Bregman divergence/distance of a sufficiently well-behaved  function (e.g. Legendre function) as the ``distance'' metric  \cite{bauschke97}. Therefore, our strategy to incorporate a mirror mapping into a fixed-point algorithm is that of considering a Bregman projection as the backward operator of a splitting method. 
	Recently, this integration is introduced for the forward-backward splitting in \cite{bui20}.} Unfortunately, we cannot use the Bregman algorithm in \cite{bui20} since it requires  cocoercivity of the forward operator, a property that is not satisfied by the MS-GNEP {in general}. Instead, we extend the forward-reflected-backward (FoRB) splitting \cite{malitsky20} since it works for (non-cocoercive non-strongly) monotone operators and it only requires one forward and one backward steps (see \cite{belgioioso20} for a discussion on the advantages of the FoRB over other fixed-point methods in generalized aggregative games). 


The main  contributions of this work are summarized as follows: 
\begin{itemize}
	\item We consider for the first time a mixed-strategy GNEP, formulated as a  monotone inclusion with special structure (Sections \ref{sec:gmi_games}-\ref{sec:NE_mon}). 
	\item We propose the Bregman forward-reflected-backward (B-FoRB) splitting me- thod, which is an extension of the FoRB  \cite{malitsky20} and uses the Bregman distance in the backward operator (Section \ref{sec:FoRB}). We show the convergence of this novel splitting method under monotonicity and Lipschitz continuity assumptions of the operators  and strong convexity of the regularizer function.
	\item We develop distributed algorithms to compute a variational mixed-strategy generalized Nash equilibrium of the considered game  based on the B-FoRB. By using a suitable regularizer function, we exploit the simplicial structure of the feasible set of mixed strategies to obtain a closed-form mirror mapping. In particular, we provide algorithms with semi-decentralized (Sections \ref{sec:BFORB_alg} and \ref{sec:GNE_alt_alg}) and distributed (Section \ref{sec:fd_alg}) structures. 
\end{itemize}
In Section \ref{sec:apps}, we formulate several problems as GMI games, i.e., the demand-side management in energy systems, an extension of networked Cournot markets, {and} the discrete-flow control problem. 
Furthermore, this section also showcases the performance of the proposed algorithms via numerical simulations.

\subsection*{Notation and definitions}
The set of real numbers is denoted by $\bb R$ whereas the extended real numbers are denoted by $\overline {\bb R} := \bb R \cup \{ +\infty \}$.   The vector of all $1$ (or $0$) with dimension $n$ are denoted by $\1_n$ ($\bm 0_n$). We omit the subscript when the dimension is clear from the context.    
The operator $\col(\cdot)$ stacks the arguments column-wise. The cardinality of a set is denoted by $|\cdot|$. The operators $\operatorname{int}$, $\operatorname{rint}$, and $\operatorname{bd}$ denote the interior, relative interior and boundary of a set.  The operator $\langle x, y \rangle$ denotes the inner product. We denote by $\|\cdot\|$ the Euclidean norm. In general, $p$-norms are denoted by $\|\cdot\|_p$. For a finite set $\mc A$, the set $\Delta(\mc A):=\{x \in \bb R_{\geq 0}^{|\mc A|}\mid \1^{\top} x = 1 \}$ is called the simplex of $\mc A$.  Let $P \succ 0$ be symmetric. For $x,y \in \bb R^n$,  $\langle x, y \rangle_P = \langle x, Py \rangle$  denotes the $P$-weighted Euclidean inner product. 
The graph of an operator $A \colon \bb R^n \rightrightarrows \bb R^n$ is denoted by $\gph (A)$. $\zer(A)$ defines the set of zeros of operator $A$, i.e., $\zer(A) = \{x \in \dom(A) \mid 0 \in A(x)\}$. An operator $A \colon \bb R^n \rightrightarrows \bb R^n$ is monotone if, for any $(x,y), (x',y') \in \gph (A)$,
$ \langle y - y', x-x' \rangle \geq 0 $ \cite[Def. 20.1]{bauschke11}, strictly monotone if, $x \neq x' \Rightarrow \langle y - y', x-x' \rangle > 0 $  and $\sigma$-strongly monotone if $ \langle y - y', x-x' \rangle \geq \sigma \| x-x'\|^2 $, with $\sigma >0$ \cite[Def. 22.1.(ii) and (iv)]{bauschke11}. 
Moreover, an operator $A \colon \bb R^n \to \bb R^n$ is Lipschitz continuous if there exists a constant $\ell_A>0$, such that, for all $x, x' \in \bb R^n$, 
$ \| A(x) - A(x')\| \leq \ell_A \|x-x'\|$ {\cite[Def. 1.47]{bauschke11}}. 
A continuously differentiable function $f \colon \bb R^n \to \bb R$ is $\sigma$-strongly convex with respect to a $p$-norm, with $\sigma >0$, if, for all $x,x' \in \dom f$,
$ f(x') \geq f(x) + \langle \nabla f(x), x'-x \rangle + \frac{\sigma}{2}\|x'-x\|_{p}^2.$
It also holds that
$  \langle \nabla f(x') -\nabla f(x), x'-x \rangle \geq {\sigma}\|x'-x\|_{p}^2.$
Additionally, $f$ is convex if the two previous inequalities hold for $\sigma =0$. 
\begin{definition}[Bregman distance {\cite[Def. 3.1]{bauschke97}}]
	\label{def:Breg_d}
	Suppose that $\varphi$ is a closed, convex, and proper function on $\bb R^n$ with $\operatorname{int}\dom(\varphi) \neq \varnothing$. Furthermore, let $\varphi$ be differentiable on $\operatorname{int}\dom(\varphi) $. The Bregman distance associated with  $\varphi$, denoted by $\dist_{\varphi} \colon \bb R^n \times \operatorname{int}\dom(\varphi) \to \overline{\bb R}$, is defined as 
	\vspace{-7pt}
	\begin{equation*}
		\dist_{\varphi}(x,y) = 
		\varphi(x) - \varphi(y) - \langle \nabla \varphi(y), x -y \rangle. 
		\vspace{-21pt}
	\end{equation*}
	\eod 
\end{definition}
\begin{definition}[Legendre function]
	\label{def:Breg_f}
	A closed, convex, and proper function $\varphi$ on $\bb R^n$ is a Legendre function if
	\begin{enumerate}
		\item $\varphi$ is essentially smooth: it is continuously differentiable on $\operatorname{int}\dom(\varphi)$, i.e., $\partial \varphi(x) =\{\nabla \varphi(x)\}$, $\forall x \in \operatorname{int}\dom(\varphi)$ and $\partial \varphi(x) = \varnothing, \forall x \in \operatorname{bd}\dom(\varphi)$ \cite[Def. 2.1, Fact 2.2]{bauschke97};
		\item $\varphi$ is essentially strictly convex: it is strictly convex on every convex subset of $\operatorname{int}\dom(\varphi)$  \cite[Def. 2.3]{bauschke97}. \eod
	\end{enumerate}
\end{definition}

\begin{example}[Euclidean regularizer]
	\label{ex:E_entropy}
	The Euclidean regularizer {$\varphi^{\rm E}\colon \bb R^n \to \bb R, \varphi^{\rm E}(x) :=  \frac{1}{2}\|x\|^2$ is a Legendre function with $\operatorname{int}\dom(\varphi)=\bb R^n$} and is $1$-strongly convex with respect to the Euclidean norm. \eod  
\end{example}
\begin{example}[Negative entropy]
	\label{ex:NE_entropy}
	The negative entropy function,  $\varphi^{\mathrm{NE}}\colon \bb R_{\geq 0}^n \to \bb R,$ $\varphi^{\mathrm{NE}}(x) := \sum_{j=1}^n x^j \ln( x^j)$, where $x^j$ denotes the $j$-th component of $x$, is a Legendre function with  $\operatorname{int}\dom(\varphi)=\bb R_{>0}^n$. This function is also $1$-strongly convex over   {$ \operatorname{rint} (\Delta(\mc X) )$} with respect to the $1$-norm   \cite[Example 2.5]{shalev11}. 
	\eod
\end{example}

\section{A class of generalized mixed-integer games}
\label{sec:gmi_games}

Let $\mc I = \{ 1,2,\dots, N\}$ be a set of agents with mixed-integer decision variables playing a generalized game. {Specifically,} for each agent $i\in \mc I$, let $a_i \in \mc A_i \subset \bb Z^{p_i}$,  denote the discrete variable, which lies in a finite space denoted by $\mc A_i$ of cardinality $|\mc A_i| = m_i$, and let $y_i \in \mc Y_i \subset \bb R^{n_i}$ denote the continuous variable. 
Furthermore, we denote the overall finite and continuous set of all agents by $\mc A = \prod_{i \in \mc I} \mc A_i$ and $\mc Y = \prod_{i \in \mc I} \mc Y_i$, respectively. We also denote  the concatenated integer and continuous decisions of all agents by $\bm a = \col((a_i)_{i \in \mc I})$ and $\bm y = \col((y_i)_{i \in \mc I})$. Occasionally, we also use the subscript $-i$ to denote the decisions of all agents except agent $i$, e.g., $\bm a_{-i} = \col((a_j)_{j \in \mc I \backslash \{i\}})$ and $\bm y_{-i} = \col((y_j)_{j \in \mc I \backslash \{i\}})$. In addition, we define $p = \sum_{i \in \mc I} p_i$, $m = \sum_{i \in \mc I} m_i$, and $n = \sum_{i \in \mc I} n_i$.

Now, let the cost function of each agent $i \in \mc I$ be denoted by $J_i\colon \bb R^{p+n} \to {\bb  R}$, which we assume has a separable structure: 
\begin{equation}
	J_i(\bm a,\bm y) = J_i^{\rm d}( a_i,\bm a_{-i}) + J_i^{\rm c}(y_i, \bm y_{-i}), \label{eq:cost}
\end{equation}
where $J_i^{\rm d} \colon \bb R^{p}  \to {\bb  R}$ is the cost function associated with the integer variables, whereas $J_i^{\rm c} \colon  \bb R^{n} \to {\bb R}$ is the cost function associated with the continuous variables. Then, each agent aims to solve its own optimization problem, i.e., 
\begin{subequations}	
	\begin{empheq}[right={\, \empheqrbrace \forall i \in \mc I,}]{align}
			\underset{(a_i,y_i) }{\min} \ \ & J_i(a_i,\bm a_{-i},y_i, \bm  y_{-i}) \label{eq:cost_f}  \\
			\operatorname{s.t.} \hspace{2pt} \ \ & (a_i,y_i) \in \mc A_i \times  \mc Y_i\\
			&g_i^{\rm d}(a_i) + g_i^{\rm c}(y_i)  \leq \theta_i,  \label{eq:loc_cons}\\
			& \sum_{j \in \mc I} \left(h_j^{\rm d}(a_j) + h_j^{\rm c}(y_j) \right)  \leq \rho, \label{eq:coup_cons}
	\end{empheq}
	\label{eq:gmi_game}%
\end{subequations}
where $J_i$ is defined in \eqref{eq:cost}, \eqref{eq:loc_cons} defines the local constraints between the integer and continuous variables of agent $i$ with {a constant vector} $\theta_i \in \bb R^{n_{\theta_i}}$ and {some functions} $g_i^{\rm d} \colon \bb R^{p_i} \to \bb R^{n_{\theta_i}}$, $g_i^{\rm c} \colon \bb R^{n_i} \to \bb R^{n_{\theta_i}}$, whereas \eqref{eq:coup_cons} defines the {shared} coupling  constraints among all agents with {a constant vector} $\rho \in \bb R^{n_\rho}$ and {some functions} $h_j^{\rm d} \colon \bb R^{p_j} \to \bb R^{n_\rho}$, $h_j^{\rm c} \colon \bb R^{n_j} \to \bb R^{n_\rho}$, {for all $j\in \mc I$}. 
Note that the considered problem captures a large class of games. When $n_i=0$ (and $n_{\theta_i}=0$), for all $i \in \mc I$, the agents in $\mc I$ play a generalized finite game \cite{sagratella17}. Moreover, without the {shared} coupling  constraints \eqref{eq:coup_cons}, it reduces further to a finite game \cite{gao20,hart01}. On the other hand, when $p_i = 0$ (and $n_{\theta_i}=0$), for all $i \in \mc I$, we have a (generalized) continuous game \cite{yi19,belgioioso17}. {Note that} we consider more general constraints \eqref{eq:loc_cons}-\eqref{eq:coup_cons} than those in \cite{sagratella19}.

The work in \cite{sagratella19} discusses some difficulties in dealing with generalized mixed-integer games. Not only solutions (GNE) might not exist in general, but also, when they exist, computing one is extremely challenging. 
Differently from \cite{sagratella19}, here we focus on the mixed strategies; {thus, we assume that the game is played repeatedly and  
	each agent randomly plays its integer (finite) decision $a_i$ based on a probability distribution over the set $\mc A_i$.} 
{By distinctively labeling each component of $\mc A_i$ with an integer, $j \in \{ 1,2,\dots, m_i\}$ and with a slight abuse of notation, let $a_i^j \in \mc A_i = \{a_i^1,\dots,a_i^{m_i}\}$, for $j =1,\dots,{m_i}$, denote the $j$-th element of $\mc A_i$. In this regard, we denote the mixed strategy (probability distribution) of agent $i$ by $x_i = \col((x_i^j)_{j=1}^{m_i}) \in \Delta(\mc A_i)$ and the mixed strategy profile of all agents by $\bm x = \col((x_i)_{i \in \mc I})$.  The variable $x_i^j$ denotes the probability that agent $i$ chooses the action $a_i^j \in \mc A_i$. Thus, we have that $x_i^1 + x_i^2 + \dots + x_i^{m_i} = 1$.  
	In the mixed-strategy extension, we consider the ``expected'' cost function, i.e., 
	\begin{equation*}
		\bbi J_i(\bm x, \bm y) =  \bbi J_i^{\rm d}(\bm x) + J_i^{\rm c}(\bm y),
		\label{eq:expected_cost}
	\end{equation*}
	where $\bbi J_i^{\rm d}(\bm x)$ denotes the expected cost function associated with the integer variables and is defined by \cite[Eq. (4)]{gao20}:
	\begin{align}
		\bbi J_i^{\rm d}(\bm x) &= \sum_{j_1=1}^{m_1} \cdots \sum_{j_n =1}^{m_n} J_i^{\rm d}(a_1^{j_1},\dots,a_n^{j_n}) x_1^{j_1}  \cdots x_n^{j_n} 
		= \langle f_i(\bm x_{-i}), x_i \rangle, \label{eq:exp_cost}
	\end{align}
	with $f_i(\bm x_{-i})$  defines how $\bm x_{-i}$ is coupled in $\bbi J_i^{\rm d}$, i.e. {$f_i(\bm x_{-i}) = \col((f_i^j(\bm x_{-i}))_{j = 1}^{m_i})$, where $f_i^j(\bm x_{-i}) \hspace{-1pt}=\hspace{-1pt} \sum_{j_1=1}^{m_1} \hspace{-1pt} \cdots \hspace{-1pt} \sum_{j_n =1}^{m_{i-1}} \sum_{j_n =1}^{m_{i+1}}  \cdots \sum_{j_n =1}^{m_n} J_i^{\rm d}(a_1^{j_1},\dots, a_n^{j_n}) \cdot x_1^{j_1}  \cdots x_{i-1}^{j_{i-1}} x_{i+1}^{j_{i+1}} \cdots x_n^{j_n}$.} Note that $f_i(\bm x_{-i})$ is the expected cost vector, i.e., each element {($f_i^j(\bm x_{-i}$))} is the expected cost when agent $i$ chooses each action in $\mc A_i$ with probability 1.  
	
	\begin{example}[Two-player zero-sum game] 
		\label{ex:MP}
		We illustrate the transformation of the original cost function to that of mixed-strategy extension via a zero-sum game, namely two-player matching-pennies \cite[Example 3]{gao20}. The decision of each player is either choosing the head (0) or the tail (1) of a coin. If both players choose the same side, then the cost value of player 1 is $J_1^{\rm d}=+ 1$, whereas that of player 2 is $J_2^{\rm d}=-1$; Otherwise, $J_1^{\rm d}= -1$ and $J_2^{\rm d}= +1$. Now, let $x_i^1$ and $x_i^2$, for $i=1,2$, denote the probability that player $i$ chooses head ($a_i=0$) and tail ($a_i=1$), respectively. The expected cost functions of the two players  are then given by:
		\begin{equation*}
			\bb J_1^{\rm d}(\bm x) = x_2^{\top} \begin{bmatrix}
				+1 & -1 \\ -1 & +1
			\end{bmatrix} x_1, \quad 
			\bb J_2^{\rm d}(\bm x) = x_1^{\top} \begin{bmatrix}
				-1 & +1 \\ +1 & -1
			\end{bmatrix} x_2.
			\vspace{-21pt}
		\end{equation*}
		\eod
	\end{example}
	
	Next, we relax the local and coupling constraints \eqref{eq:loc_cons}-\eqref{eq:coup_cons} by considering that they are satisfied in expectation, which is justified by the fact that the game is supposed to be played repeatedly, i.e.,
	\begin{align*}
		\bb E \left(g_i^{\rm d}(a_i) \right) + g_i^{\rm c}(y_i)  \leq \theta_i \Leftrightarrow  G_i^{\rm d}x_i+ g_i^{\rm c}(y_i)  \leq \theta_i,
	\end{align*}
	and
	\begin{align*}
		\sum_{j \in \mc I} \bb E \left(h_j^{\rm d}(a_j)\right) + h_j^{\rm c}(y_j)  \leq \rho  
		\Leftrightarrow   \sum_{j \in \mc I} H_j^{\rm d} x_j + h_j^{\rm c}(y_j) \leq \rho,
	\end{align*}
	where $G_i^{\rm d} = \begin{bmatrix}g_i^{\rm d} (a_i^1) & \cdots & g_i^{\rm d} (a_i^{m_i})   \end{bmatrix}$,  i.e., the row concatenation of the vectors $g_i^{\rm d}(a_i^j)$,  $j=1,\dots, m_i$, for all actions $a_i^j \in \mc A_i$. \color{black} Similarly, $H_i^{\rm d} = \begin{bmatrix}h_i^{\rm d} (a_i^1) & \cdots & h_i^{\rm d} (a_i^{m_i})   \end{bmatrix}$, for all $i \in \mc I$.
	Thus, the agents should solve
	\begin{subequations}	
		\begin{empheq}[right={\, \empheqrbrace \forall i \in \mc I,}]{align}
				\underset{(x_i, y_i) }{\min} \ & \bbi J_i^{\rm d}(x_i,\bm x_{-i}) + J_i^{\rm c}(y_i,\bm y_{-i}) \label{eq:cost_f1}  \\
				\operatorname{s.t.} \ \ & (x_i, y_i)  \in \Delta( \mc A_i)\times \mc Y_i, \\
				& G_i^{\rm d}x_i+ g_i^{\rm c}(y_i)  \leq \theta_i,  \label{eq:loc_cons1}\\
				& \sum_{j \in \mc I}  \left( H_j^{\rm d} x_j + h_j^{\rm c}(y_j) \right) \leq \rho,
				\label{eq:coup_cons1}
		\end{empheq}
		\label{eq:gmi_game1}%
	\end{subequations}
	where $\bbi J_i^{\rm d}(x_i,\bm x_{-i})$  is linear in $x_i$ as defined in \eqref{eq:exp_cost}. For each individual optimization problem \eqref{eq:gmi_game1}, let us define the local feasible set of agent $i$ as 
	\begin{align}
		{{\Omega}}_i := \left\{ (x_i,y_i)\in \Delta(\mc A_i) \times \mc Y_i \mid  G_i^{\rm d}x_i+ g_i^{\rm c}(y_i)  \leq \theta_i\right\}. \label{eq:loc_set}
	\end{align}
	Therefore, the collective global feasible set, denoted by $K$, is then defined as 
	\begin{equation}
		K \hspace{-2pt}:=\hspace{-2pt} \Big(\prod_{i \in \mc I} {{\Omega}}_i \Big) \cap \Big\{(\bm x,\bm y) \mid \sum_{j \in \mc I} \left( H_j^{\rm d} x_j + h_j^{\rm c}(y_j) \right) \leq \rho \Big\}.
		\label{eq:glob_set}
	\end{equation}
	
	
	The mixed-strategy extension of a GMI game is a generalized continuous game with special constraint sets, i.e., some decisions are defined over a simplex. Now, we state some assumptions on the cost functions and the constraints of the game that are commonly considered in continuous GNEPs \cite{belgioioso17,franci20,yi19}. 
	\begin{assumption}
		\label{as:cont}
		For each $i \in \mc I$, the functions  $J_i^{\rm c}$, $g_i^{\rm c}$, and $h_i^{\rm c}$, in \eqref{eq:gmi_game1} are {(component-wise)} convex and continuously differentiable in $y_i$. Moreover, the set $\mc Y_i$ in \eqref{eq:gmi_game1} is nonempty, compact, and convex, and the set $K$ in \eqref{eq:glob_set} satisfies Slater's constraint qualification {\cite[Eq. (27.50)]{bauschke11}.}  \eod
	\end{assumption}

	\subsection*{Problem statement}
	Let us consider Nash equilibria as the solution concept of the game in \eqref{eq:gmi_game1}. Specifically, we formally define a mixed-strategy generalized Nash equilibrium (MS-GNE) as follows.
	\begin{definition}
		\label{def:gne}
		A set of strategies $(\bm x^{\star},\bm y^{\star})$ is a mixed-strategy generalized Nash equilibrium (MS-GNE) of the game in \eqref{eq:gmi_game1} if $(\bm x^{\star},\bm y^{\star}) \in K$ and, for each $i \in \mc I$,
		$$ \bbi J_i^{\rm d}(x_i^{\star},\bm x_{-i}^{\star}) + J_i^{\rm c}(y_i^{\star},\bm y_{-i}^{\star}) \leq \bbi J_i^{\rm d}(x_i,\bm x_{-i}^{\star}) + J_i^{\rm c}(y_i,\bm y_{-i}^{\star}), $$
		for all $ (x_i,y_i) \in {{\Omega}}_i \cap \{(u,v) \mid H_i^{\rm d} u + h_i^{\rm c} (v) \leq \rho - \sum_{j \in \mc I \backslash \{i\}}(H_j^{\rm d} x_j^{\star} + h_j^{\rm c}(y_j^{\star})) \}$. \eod
	\end{definition} 
	Our objective is to solve the MS-GNE problem (MS-GNEP), that is finding an MS-GNE of the game in  \eqref{eq:gmi_game1} via distributed algorithms. 
	\color{black} 
	
	\section{{Monotone inclusion formulation of} the generalized Nash equilibrium problem with mixed strategies}
	\label{sec:NE_mon}
	
	
	
	In this section, we formulate the MS-GNEP of the game \eqref{eq:gmi_game1} as an inclusion problem, where we want to find a zero of the sum of some monotone operators. By doing so, we can resort to and build upon a suitable operator-splitting method. 
	
	To this aim, let us first define the game mapping, $F:\bb R^{m+n} \to \bb R^{m+n}$, as follows:
	\begin{equation}
		F(\bm x,\bm y) = \col(F^{\rm d}(\bm x),F^{\rm c}(\bm y)),
		\label{eq:pseudograd}
	\end{equation}
	where $F^{\rm d}(\bm x) = \col((\nabla_{x_i}\bbi J_i^{\rm d}(\bm x) )_{i \in \mc I} ) =\col((f_i(\bm x_{-i}) )_{i \in \mc I} ) $ denotes the pseudogradient associated with the mixed strategy $\bm x$  while $F^{\rm c}(\bm y) = \col((\nabla_{y_i} J_i^{\rm c}(\bm y) )_{i \in \mc I} )$ denotes the pseudogradient associated with the continuous variable $\bm y$. 
	The mapping $F$ in \eqref{eq:pseudograd} characterizes the game and it provides a sufficient condition for the existence of a GNE. Specifically, a GNE exists if $F$ is monotone \cite[Proposition 12.11]{palomar10}. 
	Therefore, we consider the following assumption regarding the monotonicity and Lipschitz continuity of the pseudogradients $F^{\rm d}$ and $F^{\rm c}$.
	
	\begin{assumption}
		\label{as:mon_F}
		The pseudogradients $F^{\rm d}$ and $F^{\rm c}$ in \eqref{eq:pseudograd} are monotone and Lipschitz continuous, with Lipschitz constant {$\ell_{F^{\rm d}}$} and {$\ell_{F^{\rm c}}$}, respectively.  \eod 
	\end{assumption}

	It follows from Assumption \ref{as:mon_F} that $F$ is maximally monotone, which is due to \cite[Proposition 20.23]{bauschke11}, and Lipschitz continuous with constant ${\ell_F} = \max(\ell_{F^{\rm d}},\ell_{F^{\rm c}})$. 
	The monotonicity of $F^{\rm d}$ is worth discussing in view of its particular structure \eqref{eq:exp_cost}. The authors of \cite{gao20} observe that {the mixed-strategy extension of a zero-sum finite game has a monotone pseudogradient.} %
	{For instance, we see from Example \ref{ex:MP} that the pseudogradient of the matching pennies game is
		$$ F^{\rm d}(\bm x) = \begin{bmatrix}
			0 & M \\ -M & 0
		\end{bmatrix}\bm x, \ \text{where } M = \begin{bmatrix}
			+1 & -1 \\ -1 & +1
		\end{bmatrix},$$
		resulting in a monotone $F^{\rm d}$. } 
	The following case also yields monotone $F^{\rm d}$. 
	
	\begin{example}
		For each $i \in \mc I$, let $J_i^{\rm d}(\bm a)$ be linearly coupled, i.e., $J_i^{\rm d}(a_i,\bm a_{-i}) =   \pi_i(a_i) + \langle \sum_{i\in \mc I} M_{i,j} a_j, a_i \rangle,$ with $\pi_i \colon \bb R^{p_i} \to \bb R$ and $M_{i,j} \in \bb R^{p_i \times p_j}$, for all $j \in \mc I$.  Therefore, 
		$$\bbi J_i^{\rm d}(\bm x)= \langle \bar \pi_i, x_i \rangle +  \sum_{j \in \mc I \backslash \{i\}} \langle M_{i,j}A_jx_j, A_ix_i \rangle,$$
		where $\bar \pi_i = \col((\pi_i(a_i^{{l}} )+ a_i^{{l} \top}M_{i,i}^\top a_i^{{l}})_{{{l}}=1}^{m_i})$ and  $A_j~=~\begin{bmatrix} a_j^1 & \cdots & a_j^{m_j}   \end{bmatrix}$, {with $a_j^l   \in \mc A_j$, for $l =1,\dots,m_i$} and all $j \in \mc I$. It follows that 
		$$F^{\rm d}(\bm x) = \col((\bar \pi_i)_{i \in \mc I} ) +  A^{\top} M^\top A\bm x,$$
		where $A = \diag((A_i)_{i\in\mc I})$ and $M = [M_{ij}]_{i,j \in \mc I}$, in which the block diagonal $M_{i,i}=0$, for all $i \in \mc I$. The pseudogradient
		$F^{\rm d}$ is monotone if
		$ \langle F^{\rm d}( \bm x) - F^{\rm d} (\bm x'), \bm x - \bm x' \rangle = (\bm x - \bm  x') A^{\top} M^\top A (\bm x - \bm x') \geq 0,$ for all $\bm x, \bm x' \in \prod_{i \in \mc I}\Delta(\mc A_i),$
		which holds true if $M^\top \succcurlyeq 0$, {e.g., when $M$ is skew-symmetric.} \eod 
	\end{example}
	
	\begin{remark}
		\label{rem:cost_trick}
		{Suppose that $J_i^{\mathrm d}(\bm a)$, for each $i \in \mc I$, is continuously differentiable.} If the pseudogradient $\col((\nabla_{a_i}J_i^{\rm d}(\bm a))_{i \in \mc I})$ is monotone, but the pseudogradient of the expected cost $\bbi J_i^{\rm d} (\bm x)$, which is $F^d(\bm x)$, is not monotone, {we can still satisfy Assumption \ref{as:mon_F} by reformulating the game}. We can introduce a continuous auxiliary variable, denoted by $v_i \in \bb R^{p_i}$, for each $i \in \mc I$, and an additional local constraint $a_i = v_i$. Therefore, each agent $i$ has  continuous decision variables $(y_i,v_i) \in \bb R^{n_i + p_i}$ with cost function {$J_i^{\rm d}(\bm v) +J_i^{\rm c}(\bm y)$, where $\bm v = \col((v_i)_{i \in \mc I})$,} whereas its integer decision variable $a_i$ has zero cost.  Section \ref{sec:dfc} provides an example of such equilibrium  problems.  \eod  
	\end{remark}
	
	{Next, we deal with the constraint in  \eqref{eq:coup_cons1}, which couples the decision variables of all agents, by dualizing it and introducing the dual variable $\lambda_i$. Moreover, in order to exploit the simplicial structure of the feasible set of  $x_i$, we dualize the constraint in \eqref{eq:loc_cons1}, which couples $x_i$ and $y_i$ by introducing the dual variable $\mu_i$. Therefore, we obtain the following Lagrangian function, for each agent $i \in \mc I$: }
	\begin{equation}
		\begin{aligned}
			\mc L_i(\bm x,\bm y,\lambda_i, \mu_i) 
			&
			= J_i^{\rm d}(\bm x) + J_i^{\rm c}(\bm y) + \langle \mu_i, G_i^{\rm d}x_i+ g_i^{\rm c}(y_i)  - \theta_i\rangle \\
			&\quad  
			-\Big\langle \lambda_i, \rho -\sum_{j \in \mc I} \left(H_j^{\rm d}x_j + h_j^{\rm c}(y_j) \right) \Big\rangle.
		\end{aligned}
		\label{eq:Lagr1}
	\end{equation}
	Considering the Lagrangian function \eqref{eq:Lagr1}, Assumption \ref{as:cont}, and Definition \ref{def:gne}, the Karush-Kuhn-Tucker (KKT) conditions of a GNE, denoted by $(\bm x^{\star},\bm y^{\star})$, for each $i \hspace{-2pt} \in \hspace{-2pt} \mc I$, are as follows:
	\begin{equation}
		\begin{aligned}
			\bm 0 &\in \begin{bmatrix}
				\nc_{\Delta(\mc A_i)}(x_i^{\star}) \\
				\nc_{\mc Y_i}(y_i^{\star})
			\end{bmatrix}
			\hspace{-2pt}+\hspace{-2pt} \begin{bmatrix}
				\nabla_{x_i}\bbi J_i^{\rm d}(\bm x^{\star})\\
				\nabla_{y_i} J_i^{\rm c}(\bm y^{\star})
			\end{bmatrix} \hspace{-2pt}+\hspace{-2pt}  
			\begin{bmatrix}
				G_i^{\rm d \top} \\ \nabla_{y_i} g_i^{\rm c}(y_i^{\star})^{\top}
			\end{bmatrix} \mu_i^{\star} 
			+  
			\begin{bmatrix}
				H_i^{\rm d \top} \\ \nabla_{y_i} h_i^{\rm c}(y_i^{\star})^{\top}
			\end{bmatrix} \lambda_i^{\star}, \\
			\bm 0 &\in \nc_{\bb R^{n_{\theta_i}}_{\geq 0}}(\mu_i^{\star}) + \theta_i- (G_i^{\rm d} x_i^{\star} + g_i^{\rm c} (y_i^{\star})),\\
			\bm 0 &\in \nc_{\bb R^{n_\rho}_{\geq 0}}(\lambda_i^{\star}) + \rho - \sum_{j\in \mc I}(H_j^{\rm d} x_j^{\star} + h_j^{\rm c} (y_j^{\star})).
		\end{aligned}
		\label{eq:kkt1}
	\end{equation}

	As motivated in \cite{facchinei10,belgioioso17}, we are particularly interested in computing a variational GNE, where the dual variable $\lambda_i$, for all $i \in \mc I$, are equal, i.e., $\lambda_i = \lambda_j = \lambda$, for all $i, j\in \mc I$.
	Therefore,  a point $(\bm x, \bm y)$ is a variational MS-GNE if there exists $(\bm \mu, \lambda)$, where $\bm \mu = \col((\mu_i )_{i\in\mc I})$, such that $(\bm x, \bm y, \bm \mu, \lambda)$ satisfies the KKT conditions \eqref{eq:kkt1}. {Thus, by denoting  ${\bm{\omega}} = (\bm x, \bm y, \bm \mu, \lambda)$ and compactly writing \eqref{eq:kkt1}, we describe the MS-GNEP of the game \eqref{eq:gmi_game1}  as the problem of finding a zero:
		\begin{equation}
			\bm 0 \in (\mc T_1 + \mc T_2 + \mc T_3 + \mc T_4)({\bm{\omega}}),
			\label{eq:mon_inc_MSGNEP}
		\end{equation}
		where 
		{$\mc T_1\colon \mc X \times {\mc Y} \times \bb R_{\geq 0}^{n_\theta +  n_\rho} \to \bb R^{n_{\bar \omega}}$, $\mc T_2 \colon \bb R^{n_{\bar \omega}}\to \bb R^{n_{\bar \omega}}$, $\mc T_3 \colon \bb R^{m+n} \times  \bb R_{\geq 0}^{n_\theta} \times \bb R^{ n_\rho}\to \bb R^{n_{\bar \omega}}$, and $\mc T_4 \colon \bb R^{m+n} \times  \bb R^{n_\theta} \times \bb R_{\geq 0}^{ n_\rho} \to \bb R^{n_{\bar \omega}}$, with $n_{\bar \omega} = m+n+n_\theta+n_\rho$, are defined by}
		\begin{equation}
			\begin{cases}
				{\mc T}_1({\bm \omega}) := \nc_{\mc X}(\bm x) \times \nc_{\mc Y}(\bm y) \times \nc_{\bb R^{n_\theta}_{\geq 0}}(\bm \mu) \times \nc_{\bb R^{Nn_\rho}_{\geq 0}}( \lambda), \\
				{\mc T}_2({\bm \omega}) := \col(F^{\rm d}(\bm x),F^{\rm c}(\bm y), \bm \theta, 
				\rho), \\
				{\mc T}_3({\bm \omega}) := \col(G^{\rm d\top}\bm \mu,\nabla g^{\rm c}(\bm y)^{\top}\bm \mu, -G^{\rm d}\bm x- g^{\rm c}(\bm y),  \0_{n_\rho}), \\
				{\mc T}_4({\bm \omega}) := \col({H}^{\rm d\top} \lambda,  {\nabla h}^{\rm c}(\bm y)^{\top} \lambda,\bm 0_{n_\theta} ,   -{H}^{\rm d}\bm x - {h}^{\rm c}(\bm y)), 
			\end{cases}
			\label{eq:kkt_op1} %
		\end{equation}
		with }  $\mc X = \prod_{i \in \mc I} \Delta(\mc A_i)$ and $n_\theta = \sum_{i \in \mc I} n_{\theta_i}$. Moreover, $G^{\rm d } = \diag((G_i^{\rm d })_{i \in \mc I})$, $g^{\rm c} (\bm y)=  \col((g_i^{\rm c}(y_i))_{i \in \mc I}) $, $\nabla g^{\rm c}(\bm y) \hspace{-1.5pt} =\hspace{-1.5pt}  \diag(\hspace{-1pt} (\nabla_{y_i} g^{\rm c}_i(y_i)\hspace{-1pt} )_{i \in \mc I}\hspace{-1pt} )$,  $\bm \theta \hspace{-2pt} =\hspace{-1.5pt}  \col(\hspace{-1pt} (\theta_i)_{i \in \mc I}\hspace{-1pt} ) $, $H^{\rm d \top } \hspace{-2pt} =\hspace{-2pt}  \col((H_i^{\rm d \top})_{i \in \mc I})$, $h^{\rm c}(\bm y) = \sum_{i\in\mc I}h_i^{\rm c}(y_i)$, and $\nabla h^{\rm c}(\bm y)^{\top} =\col((\nabla_{y_i} h_i^{\rm c}(y_i)^{\top})_{i \in \mc I})$. 
	\section{Equilibrium seeking via Bregman operator splitting}
	\label{sec:GNE_alg}
	%
	{
		
		In this section, we design fixed-point algorithms to find $\bm \omega$ such that \eqref{eq:mon_inc_MSGNEP} holds based on a novel splitting method namely the Bregman forward-reflected-backward (B-FoRB), which is an extension of the FoRB introduced in \cite{malitsky20}. Therefore, first we present the B-FoRB splitting and afterwards we propose several distributed GNE seeking algorithms. 
	}
	
	\subsection{Bregman forward-reflected-backward splitting method}
	\label{sec:FoRB}
	
	In the most general setting, we consider the following monotone inclusion problem of the sum of two operators: 
	\begin{equation}
		\text{find ${\bm{\omega}}$ such that } \ \bm 0 \in (A + B)({\bm{\omega}}),
		\label{eq:mon_prob}
	\end{equation}
	where  { $A$, $B$, and $A+B$ are maximally monotone operators  \cite[Corollary 25.5]{bauschke11} in $\bb R^{n_\omega}$.} Furthermore, we also assume that $B$ is Lipschitz continuous. Note that we can transform \eqref{eq:mon_inc_MSGNEP} into \eqref{eq:mon_prob} by an appropriate splitting. 
	Problem \eqref{eq:mon_prob} can be solved by the standard FoRB algorithm, stated as follows:
	\begin{align}
		{\bm{\omega}}^{(k+1)} =& (\Id + \gamma A)^{-1} ({\bm{\omega}}^{(k)} - \gamma (2B({\bm{\omega}}^{(k)})-B({\bm{\omega}}^{(k-1)}))),
		\label{eq:FoRB}
	\end{align}
	where $\gamma$ is a {sufficiently small } positive step size \cite[Corollary 2.6]{malitsky20}. 
	
	Now,  let us define a regularization function $\varphi$ satisfying the following assumption.
	\begin{assumption}
		\label{as:phi}
		The function $\varphi$ is a separable {Legendre} function, i.e.,  $\varphi(\bm \omega) := \sum_{j =1}^{n_{\omega}} \varphi_j(\omega_j)$, where {each} {$\omega_j \in \bb R$ and } $\varphi_j$ is {a Legendre  function} (Definition \ref{def:Breg_f}). \eod %
	\end{assumption}
	\color{black}
	Then, we introduce the B-FoRB algorithm as follows:
	{\begin{align}
			{\bm{\omega}}^{(k+1)} =& (\nabla \varphi + \Gamma A)^{-1} \label{eq:BreFoRB} 
			(\nabla \varphi ({\bm{\omega}}^{(k)}) - \Gamma(2B({\bm{\omega}}^{(k)})-B({\bm{\omega}}^{(k-1)}))),
	\end{align}}%
	where $\Gamma \in \bb R^{{n_\omega} \times {n_\omega}}$ is a diagonal positive definite matrix, i.e., $\Gamma=\diag((\gamma_i )_{i=1}^{{n_\omega}})$. {We obtain the B-FoRB in \eqref{eq:BreFoRB} by incorporating of the gradient of the {Legendre} function $\nabla \varphi$, which replaces the identity operator $\Id$ in \eqref{eq:FoRB}. Note that, we can recover the standard FoRB as a special case whenever $\varphi (\bm \omega)= \frac{1}{2} \|\bm \omega \|^2$ (Example \ref{ex:E_entropy}).  }

	Next, we show the convergence of the sequence generated by the B-FoRB in \eqref{eq:BreFoRB}.
	\begin{theorem}[Convergence of B-FoRB]
		\label{th:conv_BreFoRB}
		Let $A$, 
		$B$, {and $A+B$} 
		be maximally monotone and $B$ be ${\ell_B}$-Lipschitz. {Let $\varphi$ {satisfy Assumption \ref{as:phi}} and be $\sigma_{\varphi}$-strongly convex on $\operatorname{int}\dom(\varphi) \cap \dom (A)$}. Furthermore, suppose that  $\zer(A + B)\cap \operatorname{int} \dom  (\varphi) \neq \varnothing$. If $\Gamma =\diag((\gamma_i )_{i=1}^{{n_\omega}})$, where $\gamma_i \in ({0,}\frac{\sigma_{\varphi}}{2 {\ell_B}}) $, for $i =1,\dots{n_\omega}$, and ${\bm{\omega}}^{(0)}= {\bm{\omega}}^{(-1)} \in \operatorname{int}\dom (\varphi)$, the sequence $({\bm{\omega}}^{(k)})_{k \in \bb N}$ generated by \eqref{eq:BreFoRB} is well defined in $\operatorname{int} \dom (\varphi) \cap \dom (A)$ and converges to a point in $\zer(A + B) \cap \operatorname{int} \dom  (\varphi)$. \eod
	\end{theorem}
	\begin{proof}
		See Appendix \ref{ap:pf:th:conv_BreFoRB}.
	\end{proof}
	\begin{remark} 
		The convergence result in Theorem \ref{th:conv_BreFoRB} holds for general inclusion problems  given by \eqref{eq:mon_prob}. In the next subsections, we use the B-FoRB method to solve the MS-GNEP \eqref{eq:mon_inc_MSGNEP}, which is a special (structured) instance of \eqref{eq:mon_prob}. \eod
	\end{remark}
	
	\subsection{Generalized Nash equilibrium seeking method based on the Bregman forward-reflected-backward {splitting}}
	\label{sec:BFORB_alg}
	{In this section,} we  present a {semi-decentralized} algorithm based on the B-FoRB. 
	Let us recall the inclusion \eqref{eq:mon_inc_MSGNEP} and consider the following splitting:
	\begin{align}
		A = \mc T_1, \quad
		B = \mc T_2 + \mc T_3 + \mc T_4, \label{eq:op_split}
	\end{align} 
	to obtain the form in \eqref{eq:mon_prob}. As discussed in Section \ref{sec:FoRB}, the B-FoRB splitting method requires maximal monotonicity of $A$ and $B$ as well as Lipschitz continuity of $B$. The operator $A=\mc T_1$ is a concatenation of normal cones of closed convex sets, thus it is maximally monotone \cite[Example 20.26, Proposition 20.23]{bauschke11}. In order to meet the requirement for $B$, we impose the following additional assumption.
	\begin{assumption}
		\label{as:L_T34}
		The operators $\mc T_3$ and $\mc T_4$ in \eqref{eq:kkt_op1} are Lipschitz continuous with positive constants {$\ell_{\mc T_3}$} and {$\ell_{\mc T_4}$}, respectively.   \eod 
	\end{assumption}

	\begin{example} 
		\label{ex:lin_constd}
		Let us discuss a special case for which Assumption \ref{as:L_T34} holds, a case that is commonly considered in the continuous GNEP literature \cite{yi19,franci20,belgioioso20} and appears in various practical applications, such as Cournot games  \cite{yi19} and energy systems \cite{belgioioso22}. 
		Let the functions $g_i^{\rm c}$ and $h_i^{\rm c}$, for all $i \in \mc N$, be linear, i.e., $g_i^{\rm c} \colon y_i \mapsto G_i^{\rm c} y_i$ and $h_i^{\rm c} \colon y_i \mapsto H_i^{\rm c} y_i$, for some matrices $G_i^{\rm c}\in \bb R^{n_{\theta_i} \times n}$ and $H_i^{\rm c}\in \bb R^{n_\rho \times n}$. Then,  
		\begin{align*}
			{\mc T_3} &= \col(G^{\rm d \top} \bm \mu ,
			G^{\rm c \top} \bm \mu ,
			-G^{\rm d} \bm x - G^{\rm c} \bm y,
			\0_{n_\rho}),\\
			\mc T_4 &= \col(H^{\rm d \top} \lambda ,
			H^{\rm c \top} \lambda ,
			\0_{n_\theta},
			-H^{\rm d} \bm x - H^{\rm c} \bm y),
		\end{align*}
		where $G^{\rm c } = \diag((G_i^{\rm c })_{i \in \mc I})$ and $H^{\rm c \top } = \col((H_i^{\rm c })_{i \in \mc I})$. The operators $\mc T_3$ and $\mc T_4$ are concatenations of linear operators similar to the operator $\mc B$ in \cite[Lemma 1]{franci20}. Thus, Assumption \ref{as:L_T34} is satisfied. \eod 
	\end{example}
	\begin{remark}
		\label{rem:bound_lam}
		
		The Lipschitz continuity of $\mc T_3$ and $\mc T_4$ (Assumption \ref{as:L_T34}) also holds when we impose some boundedness assumptions on the constraint functions. Specifically,  
		if we assume that $g_i^{\mathrm c}(y_i)$ and $\nabla_{y_i} g_i^{\mathrm c}(y_i)$, for all $i \in\mc I$, are bounded 
		{on $\mc Y_i$,} then $\mc T_3$ is Lipschitz continuous. Similarly, if $h_i^{\mathrm c}(y_i)$ and $\nabla_{y_i} h_i^{\mathrm c}(y_i)$ are bounded {on $\mc Y_i$,} $\mc T_4$ is also Lipschitz continuous.
		\eod	
	\end{remark}

	\begin{lemma}
		\label{le:Lips_B}
		Let  the operators 
		$\mc T_2$, $\mc T_3$, and $\mc T_4$ be defined in \eqref{eq:kkt_op1}. Let Assumptions \ref{as:cont}, \ref{as:mon_F}, and \ref{as:L_T34} hold. Then, the operator $B = \mc T_2 + \mc T_3 + \mc T_4$ is maximally monotone 
		and {$\ell_B$}-Lipschitz with ${\ell_B} = \max({\ell_F},{\ell_{\mc T_3}},{\ell_{\mc T_4}})$. \eod
	\end{lemma}
	\begin{proof}
		See Appendix \ref{ap:pf:le:Lips_B}. 
	\end{proof}
	
	After splitting the operators, we need to define a suitable Legendre function $\varphi({\bm{\omega}})$. Here, we exploit the fact that the feasible set of each $x_i$ is a simplex. Specifically, we consider the negative  entropy {(Example \ref{ex:NE_entropy}) as the regularizer of $\bm x$ and the squared Euclidean norm (Example \ref{ex:E_entropy}) as the regularizer of the other variables $(\bm y, \bm \mu, \lambda)$, as formally stated in the following lemma.}
	
	\begin{lemma}
		\label{le:str_cvx_Bregf}
		The function
		\begin{equation}
			\varphi({\bm{\omega}}) =
			\varphi^{\mathrm{E}}(\lambda) + \sum_{i \in \mc I}\left(\varphi_i^{\mathrm{NE}}(x_i)+\varphi_i^{\mathrm{E}}(y_i) +\varphi_i^{\mathrm{E}}(\mu_i)\right), 
			\label{eq:Breg_f_gmi}
		\end{equation}
		where $\varphi^{\rm E}(\cdot)=\frac{1}{2}\| \cdot \|^2$ and $\varphi_i^{\mathrm{NE}}(x_i) = 
		\sum_{j =1}^{m_i}  x_i^j \ln (x_i^j)$,  for all $i \in \mc I$, with $x_i^j$ denoting the $j$-th component of $x_i$, 
		is a Legendre function with domain $\bb R_{\geq 0}^m \times \bb R^{n + n_\theta + n_\rho}$ and it is $1$-strongly convex on $\operatorname{rint}(\mc X) \times \bb R^{n + n_\theta + n_\rho}$ with respect to the Euclidean norm. \eod
	\end{lemma}
	\begin{proof}
		See Appendix \ref{ap:pf:le:str_cvx_Bregf}.
	\end{proof}

	We are now ready to present a GNE-seeking algorithm based on the B-FoRB as shown in Algorithm \ref{alg:BFoRB}, { which is a {semi-decentralized} algorithm {that requires} a central coordinator, whose role is only to update the dual variable $\lambda$. 
	{For simplicity,} we assume that the communication network  coincides with the interference network, which represents the cost couplings. Additionally, In Algorithm \ref{alg:BFoRB}, $\accentset{\circ}{x}_{i} $, $\accentset{\circ}{y}_{i} $, and $\accentset{\circ}{\mu}_{i} $, for all $i \in \mc I$, as well as $\accentset{\circ}{\lambda} $ are auxiliary variables introduced only to improve the readability of the algorithm.
	
	\begin{algorithm}[!t]
		\caption{B-FoRB GNE-seeking algorithm}
		\label{alg:BFoRB}
		\textbf{Initialization}\\
		For each $i \in \mc I$, set $x_{i}^{(0)} = x_{i}^{(-1)} \in {\operatorname{rint}( \Delta(\mc A_i))}$, $y_{i}^{(0)} = y_{i}^{(-1)} \in {\bb R^{n_i}}$, and $\mu_{i}^{(0)} = \mu_{i}^{(-1)} \in \bb R^{n_{\theta_i}}$. For the central coordinator, set  $\lambda^{(0)} = \lambda^{(-1)} \in \bb R^{n_\rho}$.
		
		\textbf{Iteration until convergence}\\
		\textbf{Each agent $i \in \mc I$ :}
		\begin{enumerate}
			\item Computes $\nabla_{x_i} \bbi J_i^{\rm d} (\bm x^{(k)})$ and $\nabla_{y_i} J_i^{\rm c} (\bm y^{(k)})$.
			\item Updates $x_{i}^{(k+1)}=((x_i^1)^{(k+1)},\dots,(x_i^{m_i})^{(k+1)})$ by
			{\begin{align}&\hspace{2pt}
					\accentset{\circ}{x}_{i}^{(k)} =  \nabla_{x_i} \bbi J_i^{\rm d} (\bm x^{(k)}) + G_i^{\rm d \top} \mu_{i}^{(k)} + H_i^{\rm d \top}\lambda^{(k)}, \notag\\
					&\hspace{-20pt}(x_{i}^j)^{(k+1)} = \frac{(x_{i}^j)^{(k)}  \operatorname{exp}({-\gamma_i[2\accentset{\circ}{x}_{i}^{(k)}-\accentset{\circ}{x}_{i}^{(k-1)}]_j})}{\sum_{l=1}^{m_i}(x_{i}^{l})^{(k)} \operatorname{exp}({-\gamma_i[2\accentset{\circ}{x}_{i}^{(k)}-\accentset{\circ}{x}_{i}^{(k-1)}]_{l}})}, 
					\label{eq:x_upd_alg_BFoRB}
				\end{align}
			}
			for $j=1,\dots,m_i$.
			
			\item Updates $y_{i}^{j(k+1)}$ by
			\hspace{-20pt}\begin{align}
				&\hspace{-5pt}\accentset{\circ}{y}_{i}^{(k)}  =  \nabla_{y_i} J_i^{\rm c} (\bm y^{k}) + \nabla_{y_i}g_i^{\rm c}(y_{i}^{(k)})^{\top} \mu_{i}^{(k)} 
				+ \nabla_{y_i}h_i^{\rm c}(y_{i}^{(k)})^{\top}\lambda^{(k)}, \notag \\
				&\hspace{-15pt}y_{i}^{(k+1)} = \proj_{\mc Y_i}(y_{i}^{(k)}-\gamma_i (2\accentset{\circ}{y}_{i}^{(k)}-\accentset{\circ}{y}_{i}^{(k-1)})).
				\label{eq:y_upd_alg_BFoRB}
			\end{align}
			\item Updates $\mu_{i}^{(k+1)}$ by
			\begin{align}
				& \hspace{-5pt} \accentset{\circ}{\mu}_{i}^{(k)}  =  \theta_i - G_i^{\rm d} x_{i}^{(k)} - g_i^{\rm c}(y_{i}^{(k)}), \notag\\
				&\hspace{-15pt} \mu_{i}^{(k+1)} = \proj_{{\geq 0}}(\mu_{i}^{(k)}-\gamma_i (2\accentset{\circ}{\mu}_{i}^{(k)}-\accentset{\circ}{\mu}_{i}^{(k-1)})).
				\label{eq:mu_upd_alg_BFoRB}
			\end{align}
			\item  Sends $x_{i}^{(k+1)}$ and $y_{i}^{(k+1)}$ to other agents that require them  {and $H_i^{\rm d} x_{i}^{(k+1)}$ and $h_i^{\rm c}(y_{i}^{(k+1)})$} to the central coordinator.
		\end{enumerate}
		\textbf{end}\\
		\textbf{Central coordinator:}
		\begin{enumerate}
			\item Updates ${\lambda}^{(k+1)} $:
			\begin{align}
				&\hspace{-5pt} \accentset{\circ}{\lambda}^{(k)}  =  \rho - \sum_{i \in \mc I}\left(H_i^{\rm d} x_{i}^{(k)} + h_i^{\rm c}(y_{i}^{(k)})\right), \notag\\
				&\hspace{-15pt}\lambda^{(k+1)} = \proj_{{\geq 0}}(\lambda^{(k)}-\zeta (2\accentset{\circ}{\lambda}^{(k)}-\accentset{\circ}{\lambda}^{(k-1)})).
				\label{eq:lambda_upd_alg_BFoRB}
			\end{align}
			\item Broadcasts $\lambda^{(k+1)}$ to all agents.
		\end{enumerate}
		\textbf{end}
	\end{algorithm}

	{Let us {now } formally state} the convergence property of Algorithm \ref{alg:BFoRB}.
	\begin{theorem}
		\label{th:conv_alg_BFoRB}
		Let Assumptions \ref{as:cont}, \ref{as:mon_F}, and \ref{as:L_T34} hold. Let the sequence $\bm \omega^{(k)} = (\bm x^{(k)}, \bm y^{(k)}, \bm \mu^{(k)}, \lambda^{(k)})$ be generated by Algorithm \ref{alg:BFoRB}. Suppose that  $\zer(A+B) \cap \operatorname{int}\dom (\varphi) \neq  \varnothing$, with $A$ and $B$ defined in \eqref{eq:op_split} and $\varphi(\bm \omega)$ defined in \eqref{eq:Breg_f_gmi}. If the step sizes   of Algorithm \ref{alg:BFoRB} follows $\gamma_i, \zeta \in (0,\frac{1}{2{\ell_{B}}})$, where ${\ell_B}$ is defined in Lemma \ref{le:Lips_B}, then $(\bm x^{(k)},\bm y^{(k)})$  converges to a variational MS-GNE of the game \eqref{eq:gmi_game1}. \eod
	\end{theorem}
	\begin{proof}
		See Appendix \ref{ap:pf:th:conv_alg_BFoRB}. 
	\end{proof}
	
	
	\subsection{Alternative algorithm}
	\label{sec:GNE_alt_alg}

	Differently from Algorithm \ref{alg:BFoRB}, we {can} implicitly consider the constraints \eqref{eq:loc_cons1} by using the local set ${{\Omega}}_i$ \eqref{eq:loc_set}. In this regard, we use a different KKT conditions: For all $i \in \mc I$, 
	\begin{equation}
		\begin{cases}
			\bm 0 \in \nc_{{{\Omega}}_i}(x_i^{\star},y_i^{\star}) + \begin{bmatrix}
				\nabla_{x_i}\bbi J_i^{\rm d}(\bm x^{\star})\\
				\nabla_{y_i} J_i^{\rm c}(\bm y^{\star})
			\end{bmatrix} +  
			\begin{bmatrix}
				H_i^{\rm d \top} \\ \nabla_{y_i} h_i^{\rm c}(y_i^{\star})^{\top}
			\end{bmatrix} \lambda^{\star} \\
			\bm 0  \in \nc_{\bb R^{n_\rho}_{\geq 0}}(\lambda^{\star}) + \rho - \sum_{j\in \mc I}(H_j^{\rm d} x_j^{\star} + h_j^{\rm c} (y_j^{\star})),
		\end{cases}
		\label{eq:kkt}
	\end{equation}
	for some dual variable $\lambda^{\star}\in \bb R^{n_\rho}$, as we consider a variational GNE. {By defining}  ${\bm{\omega}}' = (\bm x,\bm y,\lambda)$, we now have the following monotone inclusion problem:
	\begin{equation}
		\bm 0 \in (\mc S_1 + \mc S_2 + \mc S_3)({\bm{\omega}'}),
		\label{eq:mon_inc2}
	\end{equation}
	where {$\mc S_1\colon \Omega \times \bb R_{\geq 0}^{ n_\rho} \to \bb R^{n_{\omega'}}$, $\mc S_2 \colon \bb R^{n_{\omega'}}\to \bb R^{n_{\omega'}}$, and $\mc S_3 \colon \bb R^{m+n} \times \bb R_{\geq 0}^{ n_\rho}\to \bb R^{n_{\omega'}}$, with $n_{\omega'} = m+n+n_\rho$, are defined by}
	\begin{align*}
		\mc S_1(\bm \omega') &= \nc_{{{\Omega}}}(\bm x, \bm y) \times \nc_{\bb R^{n_\rho}_{\geq 0}}(\lambda),\quad \ \
		\mc S_2(\bm \omega') = \col(F(\bm x,\bm y), \rho), \\
		\mc S_3(\bm \omega') &= \col(H^{\rm d \top} \lambda ,
		\nabla h^{\rm c}(\bm y)^{\top} \lambda ,
		-H^{\rm d} \bm x - h^{\rm c} (\bm y)),
	\end{align*}
		with ${{\Omega}} = \prod_{i \in \mc I} {{\Omega}}_i$. By considering $A = \mc S_1$ and $B= \mc S_2 + \mc S_3$, we arrive at the formulation \eqref{eq:mon_prob}. Moreover, similarly to the operator $\mc T_4$, we also assume that $\mc S_3$ is Lipschitz continuous, which is formally stated as follows.
		\begin{assumption}
			\label{as:L_S3}
			The operator $\mc S_3$ in \eqref{eq:mon_inc2} is {$\ell_{\mc S_3}$-Lipschitz continuous.} \eod
		\end{assumption}
		
		By considering \eqref{eq:mon_inc2} and the Euclidean regularizer $\varphi^{\mathrm{E}}(\bm \omega')$ (c.f. Example \ref{ex:E_entropy}), 
		we obtain the second GNE seeking algorithm,
		which is also an instance of the standard FoRB splitting. Specifically, in this algorithm, each agent $i \in \mc I$ updates their local variables $(x_i^{(k+1)}, y_i^{(k+1)})$ by
		\begin{align}
			\xi_{i}^{(k)} =&  \begin{bmatrix}
				\nabla_{x_i} \bbi J_i^{\rm d} (\bm x^{(k)})  + H_i^{\rm d \top}\lambda^{(k)}\\
				\nabla_{y_i} J_i^{\rm c} (\bm y^{(k)})  + \nabla_{y_i}h_i^{\rm c}(y_{i}^{(k)})^{\top}\lambda^{(k)}
			\end{bmatrix}, \notag \\
			(x_{i}^{(k+1)},y_{i}^{(k+1)}) =& \  \proj_{\Omega_i}\hspace{-2pt}\left(\hspace{-2pt}\begin{bmatrix} x_{i}^{(k)} \\ y_{i}^{(k)} \end{bmatrix}-\gamma_i (2\xi_{i}^{(k)}-\xi_{i}^{(k-1)})\hspace{-2pt}\right) \hspace{-2pt},
			\label{eq:xy_upd_alg_BFoRB}
		\end{align}
		whereas the central coordinator updates the dual variable $\lambda$ by using \eqref{eq:lambda_upd_alg_BFoRB}. 
		We note that the projection step in \eqref{eq:xy_upd_alg_BFoRB}  may be more computationally expensive than the steps \eqref{eq:x_upd_alg_BFoRB}-\eqref{eq:mu_upd_alg_BFoRB} in Algorithm \ref{alg:BFoRB}. In \eqref{eq:xy_upd_alg_BFoRB}, each agent must find a point in $ (\Delta(\mc A_i) \times \mc Y_i) $ that satisfies \eqref{eq:loc_cons1}  at each iteration. On the other hand, in Algorithm \ref{alg:BFoRB}, we have an explicit formula to update $x_i$ such that it always lies in $\Delta(\mc A_i)$ and the satisfaction of  \eqref{eq:loc_cons1} is achieved asymptotically.
		\color{black}

		\begin{corollary}
			\label{th:conv_alg_BFoRB2} 
			Let Assumptions  \ref{as:cont}, \ref{as:mon_F}, and \ref{as:L_S3} hold. {Furthermore,} let the sequence $\bm \omega'^{(k)} = (\bm x^{(k)}, \bm y^{(k)}, \lambda^{(k)})$ be generated by \eqref{eq:x_upd_alg_BFoRB} and \eqref{eq:lambda_upd_alg_BFoRB} and initialized as $(x_{i}^{(0)}, y_{i}^{(0)}) = (x_{i}^{(-1)}, y_{i}^{(-1)}) \in {\bb R^{m_i + n_i}}$, for all $i \in \mc I$, and $\lambda^{(0)} = \lambda^{(-1)} \in \bb R^{n_\rho}$. Moreover, suppose that  $\zer(\mc S_1 + \mc S_2 + \mc S_3) \cap \operatorname{int}\dom (\varphi^{\mathrm{E}}) \neq  \varnothing$, where $\mc S_1$, $\mc S_2$, and $\mc S_3$ are defined in \eqref{eq:mon_inc2} and $\varphi^{\mathrm{E}}(\bm \omega')$ is defined in Example \ref{ex:E_entropy}. If the step sizes   in \eqref{eq:x_upd_alg_BFoRB} and \eqref{eq:lambda_upd_alg_BFoRB}  follows $\gamma_i, \zeta \in (0,\frac{1}{2\max(\ell_{F},\ell_{\mc S_3})})$, then $(\bm x^{(k)},\bm y^{(k)})$  converges to a variational MS-GNE of the game \eqref{eq:gmi_game1}.  \eod
		\end{corollary}
		\begin{proof}
			It follows Theorem \ref{th:conv_BreFoRB} by considering the {Legendre} function $\varphi^{\mathrm{E}}(\bm \omega')$ and letting $A = \mc S_1$ and $B = \mc S_2 + \mc S_3$. 
		\end{proof}
		
		\subsection{Distributed algorithms {without a central coordinator}}
		\label{sec:fd_alg}
		Algorithm \ref{alg:BFoRB} and its alternative proposed in Section \ref{sec:GNE_alt_alg} need a central coordinator to update the dual variable $\lambda$. However, {we can} remove {the role of a} central coordinator {by assigning} a local copy of $\lambda$ to each agent  and {adding} a consensus process into the algorithm such that these copies of $\lambda$ converge to a {common} variational dual variable $\lambda^{\star}$ \cite{yi19,franci20}. 
		First, we setup a communication network over which the consensus process of $\lambda$ is carried out. Let this network be represented by an undirected graph $\mc G^\lambda :=(\mc I, \mc E, W )$, where $\mc E \subseteq \mc I \times \mc I$ denotes the set of edges that connect the agents. Furthermore, we denote by $W = [w_{i,j}] \in \bb R^{N \times N}$ the weighted adjacency matrix of $\mc G^{\lambda}$ with $w_{i,j} >0$ if $(i,j) \in \mc E$ and by $\mc N_i^\lambda := \{j \in \mc I \mid (i,j) \in \mc E \}$ the set of neighbors of agent $i$ in $\mc G^{\lambda}$. Let us consider the following common assumption on the graph connectivity \cite{yi19,franci20}.
		\begin{assumption}
			\label{as:conG}
			The graph $\mc G^{\lambda}$ is connected and $W$ is symmetric. \eod
		\end{assumption}
		We denote the weighted Laplacian matrix of $\mc G^\lambda$ by $\L := \diag(( \sum_{j=1}^N w_{i,j} )_{i=1}^N) - W$. Due to Assumption \ref{as:conG}, $\L$ is symmetric and positive semi-definite with real and distinct eigenvalues $0=s_1 < s_2 <\dots < s_N$ \cite[Section 2.2]{yi19}.

		Focusing on Algorithm \ref{alg:BFoRB}, { as the alternative algorithm can be treated similarly},  we introduce an auxiliary variable $\nu_i \in \bb R^{n_\rho}$, for each agent $i \in \mc I$, used for the consensus process. Then, by letting the dual variable $\lambda_i$ as a local variable of each agent and by setting $\bm \lambda = \col((\lambda_i)_{i\in\mc I})$ and $\bm \nu = \col((\nu_i)_{i\in\mc I})$, we can augment the KKT inclusion used to design Algorithm \ref{alg:BFoRB}, 
		i.e., \eqref{eq:mon_inc_MSGNEP} 
		based on \cite{yi19}. Specifically, we consider 
		\begin{equation}
			\bm 0 \in (\tilde{\mc T}_1 + \tilde{\mc T}_2 + \tilde{\mc T}_3 + \tilde{\mc T}_4 + \tilde{\mc T}_5)(\tilde{\bm \omega}),
			\label{eq:kkt_op2}
		\end{equation}
		where $\tilde{\bm \omega} = (\bm x, \bm y, \bm \mu, \bm \lambda, \bm \nu)$. The operators {$\tilde{\mc T}_1\colon \mc X \times \mc Y \times \bb R_{\geq 0}^{n_\theta +  Nn_\rho} \times \bb R^{N n_\rho}\to \bb R^{n_{\tilde \omega}}$, $\tilde{\mc T}_2 \colon \bb R^{n_{\tilde \omega}}\to \bb R^{n_{\tilde \omega}}$, $\tilde{\mc T}_3 \colon \bb R^{m+n} \times \bb R_{\geq 0}^{n_\theta }  \times \bb R^{2 N n_\rho} \to \bb R^{n_{\tilde \omega}}$, and $\mc T_4 \colon \bb R^{m+n} \times \bb R^{n_\theta} \times  \bb R_{\geq 0}^{Nn_\rho}\times  \bb R^{Nn_\rho} \to \bb R^{n_{\tilde \omega}}$, with $n_{\tilde \omega} = m+n+n_\theta+2Nn_\rho$, are defined by}
		\begin{align*}
			\tilde {\mc T}_1(\tilde{\bm \omega}) &= \nc_{\mc X}(\bm x) \hspace{-1pt}\times\hspace{-1pt} \nc_{\mc Y}(\bm y) \hspace{-1pt}\times\hspace{-1pt} \nc_{\bb R^{n_\theta}_{\geq 0}}(\bm \mu) \hspace{-1pt}\times\hspace{-1pt} \nc_{\bb R^{Nn_\rho}_{\geq 0}}(\bm \lambda) \hspace{-1pt}\times\hspace{-1pt}  \bm 0_{Nn_\rho}\hspace{-1pt}, \\
			\tilde {\mc T}_2(\tilde{\bm \omega}) &= \col(F^{\rm d}(\bm x),F^{\rm c}(\bm y), \bm \theta, \1_N \otimes \rho/N  , \bm 0_{Nn_\rho}), \\
			\tilde {\mc T}_3(\tilde{\bm \omega}) &= \col(G^{\rm d\top}\bm \mu,\nabla g^{\rm c}(\bm y)^{\top}\bm \mu, -G^{\rm d}\bm x- g^{\rm c}(\bm y),  \bm 0_{2Nn_\rho}), \\
			\tilde {\mc T}_4(\tilde{\bm \omega}) &= \col(\tilde{H}^{\rm d\top}\bm \lambda, \hspace{-2pt} \tilde{\nabla h}^{\rm c}(\bm y)^{\top}\bm \lambda,\hspace{-1pt}\bm 0_{n_\theta} , \hspace{-1pt}  -\tilde{H}^{\rm d}\bm x \hspace{-2pt}-\hspace{-2pt} \tilde{h}^{\rm c}(\bm y),   \bm 0_{Nn_\rho}), \\
			\tilde {\mc T}_5(\tilde{\bm \omega}) &= \col(\bm 0_{m+n+n_{\theta}}, -(\L \otimes I_{n_\rho})\bm \nu, (\L \otimes I_{n_\rho})\bm \lambda  ), 
		\end{align*}
		where $\tilde{H}^{d} \hspace{-1pt} =\hspace{-1pt}  \diag(\hspace{-1pt} (H_i^{d})_{i\in\mc I}\hspace{-1pt} )$, $\tilde{h}(\bm y) \hspace{-1pt} =\hspace{-1pt}  \col(\hspace{-1pt} (h_i(y_i)\hspace{-1pt} )_{i\in\mc I}\hspace{-1pt} )$, and $\tilde{\nabla h}^{\rm c}(\bm y) \hspace{-1pt} =\hspace{-1pt}  \diag(\hspace{-1pt} (\nabla_{y_i}h_i^{\rm c}(y_i)\hspace{-1pt} )_{i\in\mc I}\hspace{-1pt} )$. While $\tilde{\mc T}_1$, $\tilde{\mc T}_2$, $\tilde{\mc T}_3$, and $\tilde{\mc T}_4$ are the augmentations of the operators in \eqref{eq:kkt_op1} due to that of $\bm \lambda$ and the addition of $\bm \nu$, $\tilde{\mc T}_5$ defines the consensus of $\bm \lambda$ and $\bm \nu$. Similarly to \cite[Theorem 2]{yi19}, given that $\tilde{\bm \omega}=\tilde{\bm \omega}^{\star}$ such that \eqref{eq:kkt_op2} holds, then $(\bm x^{\star},\bm y^{\star})$ is a variational MS-GNE. We omit the proof as it follows that of \cite[Theorem 2]{yi19}.

		The operators $\tilde{\mc T}_j$, $j=1,\dots,4$  
		are maximally monotone, which can be shown by following the first part of the proof of Lemma \ref{le:Lips_B} (c.f. Appendix \ref{ap:pf:le:Lips_B}). Moreover, $ \tilde{\mc T}_2$ 
		is $\ell_F$-Lipschitz continuous, similarly to $\mc T_2$. 
		Since $\tilde{\mc T}_3$ is $\col(\mc T_3, \bm 0)$, it is $\ell_{\mc T_3}$-Lipschitz continuous when Assumption \ref{as:L_T34} holds. Furthermore, we consider Assumption \ref{as:t4_s3} for $\tilde{\mc T}_4$ 
		and {assert} the maximal monotonicity and Lipschitz continuity of $\tilde{\mc T}_5$. 
		\begin{assumption}
			\label{as:t4_s3}
			The operator  $\tilde{\mc T}_4$ in \eqref{eq:kkt_op2} 
			is {$\ell_{\tilde{\mc T}_4}$-Lipschitz continuous.}
			\eod 
		\end{assumption}

		\begin{lemma}
			\label{le:cons_op}
			Let Assumption \ref{as:conG} hold. The operators $\tilde{\mc T}_5$ in \eqref{eq:kkt_op2} 
			is maximally monotone and {$\ell_{\tilde{\mc T}_5}$}\hspace{-1pt}-Lipschitz continuous, where ${\ell_{\tilde{\mc T}_5} }= 2 \max_{i \in \mc I} \sum_{j=1}^N w_{i,j}$. \eod
		\end{lemma}
		\begin{proof}
			See Appendix \ref{ap:pf:le:cons_op}.
		\end{proof}

		Next, we define the {regularizer function} that we incorporate into the inclusions \eqref{eq:kkt_op2}. 
		Indeed we use a {Legendre} function {similarly to that} for Algorithm \ref{alg:BFoRB}, i.e., 
		\begin{align}
			\tilde{\varphi}(\tilde{\bm \omega}) &= \sum_{i\in\mc I} \varphi_i^{\mathrm{NE}}(x_i) + \varphi_i^{\mathrm{E}}(y_i,\mu_{i},\lambda_i,\nu_i). \label{eq:Breg_f_gmi3}
		\end{align}  
		{Now,} we are ready to present a variant of Algorithm \ref{alg:BFoRB}  {without a central  coordinator}, as stated in Algorithm \ref{alg:BFoRB_dist}, with step sizes that can be set based on Assumption \ref{as:step_size}.

		\begin{algorithm}[!t]
			\caption{Distributed B-FoRB GNE-seeking algorithm}
			\label{alg:BFoRB_dist}
			\textbf{Initialization}\\
			For $i \in \mc I$, set $x_{i}^{(0)} = x_{i}^{(-1)} \in {\operatorname{rint}( \Delta(\mc A_i))}$, $y_{i}^{(0)} = y_{i}^{(-1)} \in {\bb R^{n_i}}$, $\mu_{i}^{(0)} = \mu_{i}^{(-1)} \in \bb R^{n_{\theta_i}}$,  $\lambda_i^{(0)} = \lambda_i^{(-1)} \in \bb R^{n_\rho}$, and $\nu_i^{(0)} = \nu_i^{(-1)} \in \bb R^{n_\rho}$.
			
			\textbf{Iteration until convergence}\\
			\textbf{Each agent $i \in \mc I$ :}
			\begin{enumerate}
				\item Computes $\nabla_{x_i} \bbi J_i^{\rm d} (\bm x^{(k)})$ and $\nabla_{y_i} J_i^{\rm c} (\bm y^{(k)})$.
				\item Updates $x_{i}^{(k+1)}$ and $y_{i}^{j(k+1)}$ by \eqref{eq:x_upd_alg_BFoRB} and \eqref{eq:y_upd_alg_BFoRB}, with $\lambda^{(k)} = \lambda_i^{(k)}$.

				\item Updates $\mu_{i}^{(k+1)}$ by \eqref{eq:mu_upd_alg_BFoRB} and 
				$\lambda_i^{(k+1)}$ by
				\begin{align}
					&\hspace{-5pt} \accentset{\circ}{\lambda}_i^{(k)}  =  \rho/N - \left(H_i^{\rm d} x_{i}^{(k)} + h_i^{\rm c}(y_{i}^{(k)})\right) 
					- \sum_{j \in \mc N_i^\lambda} w_{i,j} (\nu_i^{(k)} - \nu_j^{(k)}), \notag \\
					&\hspace{-15pt}\lambda_i^{(k+1)} = \proj_{{\geq 0}}(\lambda_i^{(k)}-\gamma_i (2\accentset{\circ}{\lambda_i}^{(k)}-\accentset{\circ}{\lambda_i}^{(k-1)})).
					\label{eq:lambda_upd_alg_BFoRB_d}
				\end{align}
				\item Updates $\nu_i^{(k+1)}$ by
				\begin{align}
					&\hspace{-5pt}	\accentset{\circ}{\nu}_i^{(k)}  = \sum_{j \in \mc N_i^\lambda} w_{i,j} (\lambda_i^{(k)} - \lambda_j^{(k)}), \notag \\
					&\hspace{-15pt}\nu_i^{(k+1)} =\nu_i^{(k)}-\gamma_i (2\accentset{\circ}{\nu_i}^{(k)}-\accentset{\circ}{\nu_i}^{(k-1)})).
					\label{eq:nu_upd_alg_BFoRB_d}
				\end{align}
				\item Sends $(x_{i}^{(k+1)},y_{i}^{(k+1)})$ to other agents that require them and $(\lambda_i^{(k+1)},\nu_i^{(k+1)})$ to the neighbors $j \in \mc N_i^\lambda$.
			\end{enumerate}
			\textbf{end}
		\end{algorithm}
			\begin{assumption}
				\label{as:step_size}
				The step size of each agent $i \in \mc I$, $\gamma_i$, in Algorithm \ref{alg:BFoRB_dist} satisfies  $\gamma_i \in (0,\frac{1}{2 \max(\ell_F,\ell_{\mc T_3},\ell_{\tilde{\mc T}_4},\ell_{\tilde{\mc T}_5})})$. 
				\eod 
			\end{assumption}
			\begin{theorem}
				\label{th:conv_alg_BFoRB_d}
				Let Assumptions \ref{as:cont}, \ref{as:mon_F}, {\ref{as:L_T34}}, \ref{as:conG}, \ref{as:t4_s3}, and \ref{as:step_size} hold. Let the sequence $\tilde{\bm \omega}^{(k)} = (\bm x^{(k)}, \bm y^{(k)}, \bm \mu^{(k)}, \bm \lambda^{(k)}, \bm \nu^{(k)})$ be generated by Algorithm \ref{alg:BFoRB_dist}. Moreover, suppose that  $\zer(\tilde{\mc T}_1 + \tilde{\mc T}_2 + \tilde{\mc T}_3 + \tilde{\mc T}_4 + \tilde{\mc T}_5) \cap (\operatorname{int}\dom (\tilde \varphi)) \neq  \varnothing$, where $\tilde{\mc T}_j$, $j=1,\dots,5$, are defined in \eqref{eq:kkt_op2} and $\tilde \varphi(\tilde{\bm \omega})$ is defined in \eqref{eq:Breg_f_gmi3}. Then the sequence $(\bm x^{(k)},\bm y^{(k)})$  converges to a variational MS-GNE of the game \eqref{eq:gmi_game1}. \eod
			\end{theorem}
			\begin{proof}
				Let us define $A = \tilde{\mc T}_1$ and $B=   \tilde{\mc T}_2 +  \tilde{\mc T}_3 +  \tilde{\mc T}_4 +  \tilde{\mc T}_5$. {Thus, Algorithm \ref{alg:BFoRB_dist} is} the expanded form of the B-FoRB algorithm \eqref{eq:BreFoRB} with $A$ and $B$ as previously defined and the Legendre function $\tilde \varphi$ defined in \eqref{eq:Breg_f_gmi3}, which satisfies Assumption \ref{as:phi} and is 1-strongly convex on $\dom(\tilde{\mc T}_1)\cap \operatorname{int}\dom(\varphi)$. The proof of this statement follows that of Lemma \ref{pr:BFoRB} in Appendix \ref{ap:pf:th:conv_alg_BFoRB}. Then, the proof {is analogous} to that of Theorem \ref{th:conv_alg_BFoRB}.
			\end{proof}
			\color{black}

			\section{Illustrative examples}
			\label{sec:apps}
			In this section, we present several problems that can be modeled as {MS-GNEPs of GMI games}. 
				Furthermore, we show the efficacy of the  algorithms proposed in Section \ref{sec:GNE_alg}  through numerical simulations. {These simulations are carried out in \textsc{Matlab} 2020b on a laptop computer with 2.3 GHz intel core i5 and 8 GB of RAM. Moreover, whenever a quadratic programming must be solved to perform a projection operation, we use the  \texttt{OSQP} solver \cite{osqp}, where we set the absolute and relative tolerances to  $10^{-10}$.} 
				
				\subsection{Demand-side management with on/off devices}
				\label{sec:DSM}
				
				{The first GMI game we discuss is a demand-side management problem in the energy domain.}  Suppose that each consumer  in an electrical distribution grid, $i \in \mc I$, owns a set of flexible devices/loads  denoted by $\mc D_i$, i.e., devices that consume energy and can be automatically operated by a local controller. Specifically, we {assume} that these devices are interruptable ones \cite{kim13}. {Thus,} each device $j \in \mc D_i$ consumes at least a certain amount of energy over a certain period of discrete time, i.e.,
				\begin{align}
					\sum_{t = 1}^T y^{j}_{i,t} &\geq E^j_i, \forall j \in \mc D_i, \quad \forall i \in  \mc I,
					\label{eq:DSM_loc_set_y}
				\end{align}
				where $T$ denotes the time period, $y^{j}_{i,t} \in \bb R$ denotes the consumption of device $j$ at time step $t$, and $E^j_i \in \bb R$ denotes the minimum energy consumed by device $j$. We suppose that the operation of each device is done by switching it on or off at each $t$ and adjusting the energy usage. {Therefore, we introduce} a binary decision, $a^{j}_{i,t} \in \{0,1\}$ that represents the state of the device (off $(0)$/on $(1)$) and the following constraint:
				\begin{align}
					\underline y_i^j a^{j}_{i,t} \leq y^{j}_{i,t} \leq  \overline y_i^j a^{j}_{i,t}, 
					\label{eq:DSM_loc_cons}
				\end{align}
				for $t=1,\dots, T$, each $j \in \mc D_i$, and $i \in \mc I$, where  $0 < \underline y_i^j \leq \overline y_i^j$ denote the lower and upper bounds of energy consumed by device $j$ when it is on. 
				Furthermore, for each consumer, we denote $ a_{i,t} = \col((a_{i,t}^j)_{j \in \mc D_i})$ and $ y_{i,t} = \col((y_{i,t}^j)_{j \in \mc D_i})$. 
				In addition to flexible devices, we also suppose that each consumer also owns inflexible ones, whose aggregated load profile during the period $T$ is assumed to be known and denoted by $P_i = \col((P_{i,t})_{t=1}^T) \in \bb R^T$. 
				In this demand-side management problem, each consumer aims to minimize its electricity cost,  defined by
				\begin{equation}
					J_i^{\rm c} = \sum_{t=1}^{T} q_t(\bm y_t) \1^{\top}  y_{i,t},
					\label{eq:DSM_cost}
				\end{equation}
				where $\bm y_t = \col((y_{i,t})_{i \in \mc I})$, for $t=1,\dots, T$ and $q_t(\cdot)$ is the electricity {price} that varies based on the total consumption of the whole network \cite{atzeni13}, i.e., for $t=1,\dots, T$, 
				$$q_t(\bm y_t) = r_t\sum_{i \in \mc I} \left(P_{i,t} + \1^{\top}  y_{i,t}\right), $$
				where $r_t >0$, for $t=1,\dots, T$. 
				Additionally, in order to avoid congestion, we limit the total consumption of the distribution grid, i.e.,
				\begin{equation}
					\underline P^{\rm grid} \leq \sum_{i\in \mc I}\left(P_{i,t} + \1^{\top}  y_{i,t}\right) \leq \overline P^{\rm grid},
					\label{eq:DSM_coup_cons}
				\end{equation}
				for $t=1,\dots, T$, where $\underline P^{\rm grid}$ and $\overline P^{\rm grid}$ denote the lower and upper bound of distribution grid total consumption. 
				
				Therefore, we obtain a GMI game in the form of Problem \eqref{eq:gmi_game} where $a_i = \col((a_{i,t})_{t=1}^T)$, $y_i = \col((y_{i,t})_{t=1}^T)$, $\mc A_i = \prod_{t=1}^T \prod_{j\in \mc D_i} \{0,1\}$, $\mc Y_i$ is defined by \eqref{eq:DSM_loc_set_y}, the cost function \eqref{eq:cost_f} is defined by $J_i^{\rm c}$ in \eqref{eq:DSM_cost} and $J_i^{\rm d} = 0$. {Moreover, the constraints \eqref{eq:loc_cons} and \eqref{eq:coup_cons} are defined by  \eqref{eq:DSM_loc_cons} and \eqref{eq:DSM_coup_cons}, respectively.} 
				{Now, we formulate the corresponding mixed-strategy GMI game. By assuming that each device works independently, we can consider that each $a_{i,t}^j$ is randomly drawn with mixed-strategy  $x_{i,t}^j \in \Delta(\{0,1\})$. 
					In this regard, we can relax the constraint \eqref{eq:DSM_loc_cons} as follows:
					\begin{align}
						\bb E (\underline y_i^j a^{j}_{i,t}) \leq y^{j}_{i,t} \leq  \bb E(\overline y_i^j a^{j}_{i,t})
						\Leftrightarrow \ 
						\underline y_i^j \begin{bmatrix}
							0 & 1
						\end{bmatrix} x^j_{i,t}  \leq y^{j}_{i,t} \leq \overline y_i^j \begin{bmatrix}
							0 & 1
						\end{bmatrix} x^j_{i,t},
						\label{eq:DSM_loc_cons1}
					\end{align}
					for all $j \in \mc D_i$ and $t\in \{1,\dots, T\}$, where $x^j_{i,t}=(x^j_{i,t}(a_{i,t}^j=0),x^j_{i,t}(a_{i,t}^j=1))  \in \bb R^2$ denotes the probability distribution of the state of device $j \in \mc D_i$. A practical interpretation of $x^j_{i,t}$ is the durational proportion of the state of device $j$ at each discrete time $t$. Then, by defining $x_i = \col((\col((x_{i,t}^j)_{t=1}^T))_{j\in \mc D_i})$, we can have a mixed strategy GMI game defined by:
					\begin{equation*}	
						\begin{rcases}
							\underset{(x_i, y_i) }{\min} &\hspace{-8pt} J_i^{\rm c}(y_i,\bm y_{-i})   \\
							\operatorname{ s.t.  } &\hspace{-8pt}  (x_i, y_i)  \in \mc X_i \times \mc Y_i, \\
							&\hspace{-8pt}  \text{\eqref{eq:DSM_coup_cons} and \eqref{eq:DSM_loc_cons1},} \ \forall j \in \mc D_i, t =1,\dots, T,
						\end{rcases}
						\forall i \in \mc I,
					\end{equation*}
					where $\mc X_i = \prod_{j\in \mc D_i} \prod_{t=1}^T \Delta(\{0,1\})$.} 
				Note that the pseudogradient of this game is monotone and Lipschitz continuous.

				As an example, we simulate a small-scale network with 10 agents, each of which owns 1 flexible device. We consider a one-day horizon with hourly sampling time, i.e., $T=24$. Furthermore, we arbitrarily assign a single or multiple household load profile \cite{jasm} to each agent as its inflexible load. The parameters of the problem are defined as follows. With uniform distribution, we sample $E_i^j \sim [160, 1000]$ W, $\underline y_i^j \sim [1, 18]$ W, $\overline y_i^j \sim [30, 180]$ W, for all $i \in \mc I$. Moreover, $ \underline P^{\rm grid}=0$, $\overline P^{\rm grid} =24$ kW and $r_t = 0.1,$ for all $t=1,\dots,T$. 
				We compute a variational MS-GNE of this problem using Algorithm \ref{alg:BFoRB}. 
				The top plot of Figure \ref{fig:dsm} shows the probability that each agent turning on its flexible device at each time step and the bottom plot of Figure \ref{fig:dsm} shows the electricity price $q_t(\bm y_t)$ at each time step. Since the network consists of households, the peak hours occur at $t={7-9}$, $t={13-15}$, and $t={18-22}$.  From the top plot of Figure  \ref{fig:dsm}, indeed we can see that each agent decides to turn off its flexible device during these peak hours, when the electricity price is high.
				\begin{figure}
					\centering
					\hspace{-15pt}\includegraphics[height=0.3\textheight]{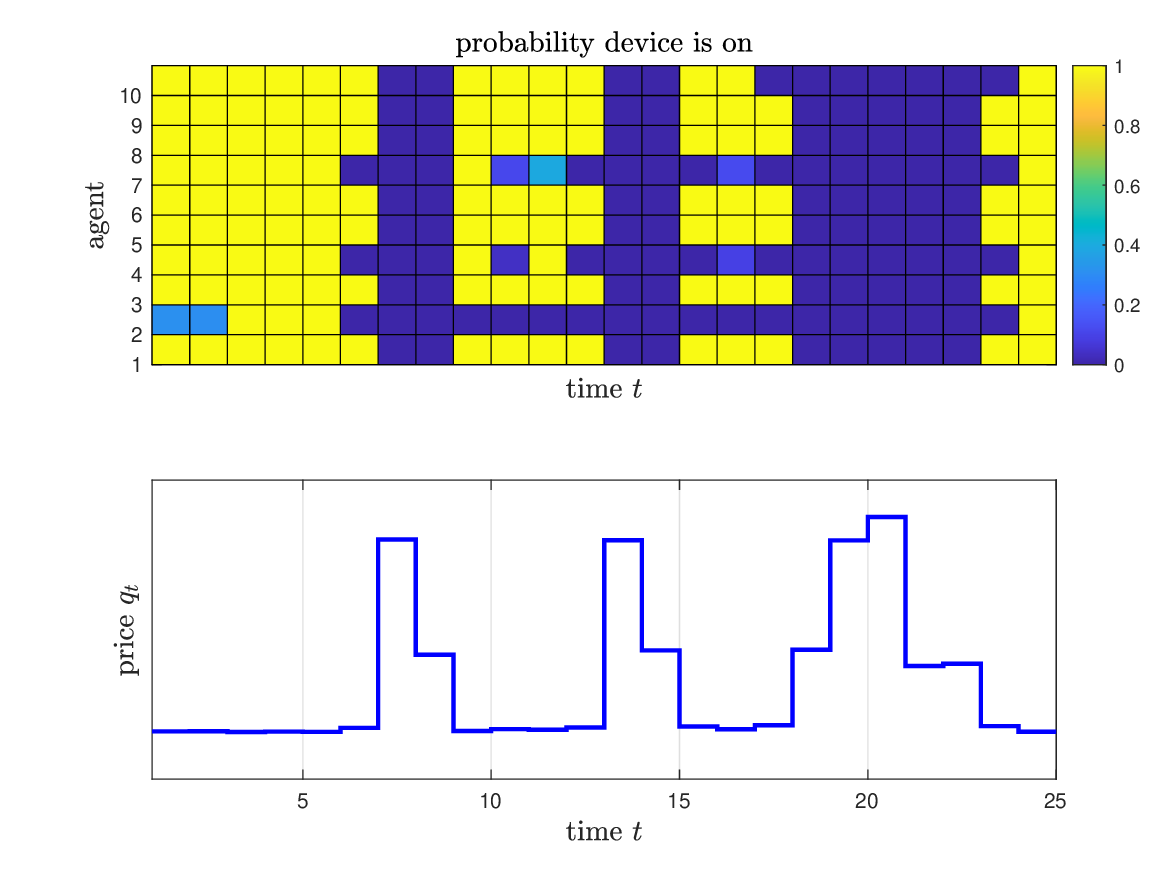}
					\caption{Simulation results of a 10-agent demand-side management. The top plot shows the probability that the device of each agent is on and the bottom plot shows the electricity price profile.}
					\label{fig:dsm}
				\end{figure}
				
				\color{black}

				\subsection{Networked Cournot games}
				\label{sec:Cour}
				{We consider that} a group of agents (firms) $ \mc I$, which produce a type of homogeneous and {divisible} good,  competes in $M$ markets based on the {Nash-}Cournot model \cite{yi19,bimpikis19,koshal16}.  In this market, a bipartite graph that determines to which subset of markets each agent can supply is given. {That is,} the links, which connect an agent and a market, are predetermined. Then, each agent decides how to allocate its product to the connected markets. This networked Cournot game can model the markets of electricity, gas, communication networks, airlines, and cement industry \cite{bimpikis19,koshal16}. 
				Here, we consider an extension where the market graph, which shows to which subset of markets each agent can supply, is not predetermined. Instead, each agent decides in which markets that it wants to participate and allocates its output at the same time. 
				
				Firstly, we {denote} the decision variables of market participation of agent $i$ {by} $z_i^\nu \in \{0,1\}$, for all $\nu \in \{1,2,\dots,M\}$, where $z_i^\nu = 1$ means that agent $i$ participates in market $\nu$. Moreover, we denote the allocated output of agent $i$ for market $\nu$ by $y_i^\nu \in \bb R$. Furthermore, we assume that if agent $i$ decides to participate in market $\nu$, it must deliver at least a certain amount of output, denoted by $\underline y^\nu$. We can consider it as a minimum bid or commitment of agent $i$ to market $\nu$. Moreover, each agent also has a production limit, denoted by $\overline y_i$. Therefore, we can translate these requirements as the following constraints:
				\begin{align}
					z_i^\nu \underline y^\nu \leq y_i^\nu &\leq z_i^\nu \overline y_i,  \label{eq:Cour_loc_cons}\\
					0 \leq \sum_{\nu=1}^M y_i^\nu &\leq \overline y_i. \label{eq:Cour_loc_set_y}
				\end{align}
				Furthermore, we may also consider $z_i^\nu$ as an additional limited resource, i.e., agent $i$ can only afford to connect to at most $\overline \nu_i$ markets, where $\overline{\nu}_i \leq M$, implying the following additional constraint:
				\begin{align}
					\sum_{\nu = 1}^M z_i^\nu \leq \overline \nu_i.  \label{eq:Cour_loc_set_a}
				\end{align}
				Secondly, each market $\nu$ has a maximum capacity, denoted by $\overline Y^\nu$. Therefore, we also have the following coupling constraints:
				\begin{align}
					\sum_{i \in \mc I} y_i^\nu \leq \overline Y^\nu, \quad \forall \nu=\{1,\dots,M\}.
					\label{eq:Cour_coup_cons}
				\end{align}
				Finally, the objective function of each agent depends on the aggregative term $\sum_{i\in \mc I} y_i$, where $y_i = \col((y_i^\nu)_{\nu=1}^M)$. Specifically, we follow \cite[{Eq. (36)}]{yi19}   that 
				\begin{equation}
					J_i^{\rm c}(\bm y) = c_i(y_i) + y_i^{\top} D y_i - \bar P^{\top} y_i +  \sum_{j \in \mc I \backslash \{i\}} y_j^{\top} D  y_i,
					\label{eq:Cour_cost}
				\end{equation}
				for all $i \in \mc I$, where $c_i(\cdot)$ is a strongly convex production cost function whereas $D = \diag((d_\nu)_{\nu=1}^M)$, $d_\nu > 0$, for $\nu=1,2,\dots,M$, and $\bar P \in \bb R^M_{>0}$ define the price function. {On the other hand,} we do not penalize the formation of connection between the agents and the markets, i.e., $J_i^{\rm d}=0$.  
				Thus, we obtain Problem \eqref{eq:gmi_game} where $a_i = \col((z_{i}^\nu)_{\nu=1}^M)$, $\mc A_i$ defined by \eqref{eq:Cour_loc_set_a}, i.e., $\mc A_i=\{a \in \{0,1\}^M \mid \1^{\top}a \leq \overline \nu_i\}$, $\mc Y_i$ defined by \eqref{eq:Cour_loc_set_y}, the cost function \eqref{eq:cost_f} is defined by $J_i^{\rm c}$ in \eqref{eq:Cour_cost} and $J_i^{\rm d} = 0$, local constraint \eqref{eq:loc_cons} defined by \eqref{eq:Cour_loc_cons}, and coupling constraints \eqref{eq:coup_cons} defined by \eqref{eq:Cour_coup_cons}. We can then obtain its mixed-strategy extension in the form \eqref{eq:gmi_game1}. 
				The pseudogradients $F^{\rm c}$ and $F^{\rm d}$ are monotone and Lipschitz continuous. In fact, $F^{\rm c}$ is strongly monotone \cite[{Section 7.1}]{yi19}.
				
				\begin{figure}
					\centering
					\includegraphics[scale=0.55]{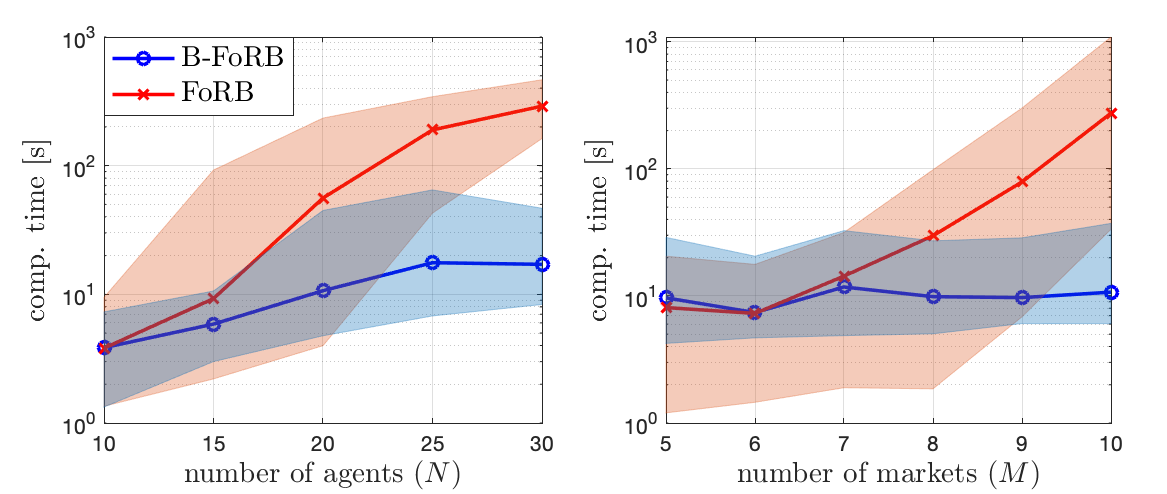}
					\caption{The total computational time/agent of the B-FoRB and FoRB with different number of agents $N$ {(left)} and different number of markets $M$ {(right)}. Each data point in the solid lines is the average of 30 Monte-Carlo simulations (shown as shaded areas).}
					\label{fig:var_nM}
				\end{figure}
				
				We showcase the performance of the B-FoRB (Algorithm \ref{alg:BFoRB})  by comparing it with the standard FoRB (c.f. Section \ref{sec:GNE_alt_alg}) using this problem. To this end, we carry out two batches of simulations. We vary the number of agents for the first batch, i.e, $N=10, 15, 20, 25, 30$, and fix $M=9$. On the other hand, we vary the number of markets ($M$) for the second batch, i.e.,  $M= 5, 6, \dots, 10$, and fix $N=20$. For each $N$ in the first batch (and each $M$ in the second batch), we carry out 30 Monte Carlo simulations. {Moreover,} for each simulation, we randomly set $\overline{y}^i, c_i$, for all $i \in \mc I$,  $\overline{Y}^\nu$ for all $\nu \in \{1,2,\dots,M\}$, $D$, and $\bar{P}$ based on the rules provided in \cite[{Section 7.2}]{yi19}. Similarly, $\overline{\nu}_i$, for each $i \in \mc I$, is randomly chosen from $[1,M]\cap \bb Z$ and $\underline{y}^\nu$, for each $\nu \in \{1,2,\dots,M\}$, is randomly chosen from $[0.05, 0.1]$. Furthermore, the step sizes $\gamma_i$ and $\zeta$ of both algorithms are set to be equal and the stopping criterion of both algorithms  is defined by $\|\bm  r^{(k)} \|_{\infty} \leq \varepsilon$, where $\varepsilon=10^{-5}$ and $\bm r^{(k)}=\bm \omega^{(k+1)} -\bm \omega^{(k)}$ for the B-FoRB whereas $\bm r^{(k)}=\bm \omega'^{(k+1)} -\bm \omega'^{(k)}$ for the FoRB. 
				{Based on the simulations, both the B-FoRB and the FoRB require similar  numbers of iterations to reach the stopping criterion. However, the advantage of the B-FoRB is on the computational time required to perform each iteration, which is lower than that of the FoRB,  {since the B-FoRB uses a closed-form formula of the mirror map onto a simplex \eqref{eq:x_upd_alg_BFoRB}.
					As shown in Figure \ref{fig:var_nM}, we can observe that the B-FoRB scales better with the number of agents and markets than the FoRB.} 

				\color{black}

				\subsection{Discrete-flow control problem} 
				\label{sec:dfc} 
				{Let us now} consider a discrete-flow control problem, which is formulated as a generalized potential integer games in \cite{sagratella17}, and discuss how our approach can complement the Gauss-Southwell (GS) algorithm \cite[Algorithm 1]{sagratella17}, which computes an $\varepsilon$-approximate equilibrium of the game \cite[Eq. (3)]{sagratella17}. Specifically, we can obtain an initial point for the GS algorithm by computing a variational MS-GNE of the game using {Algorithm \ref{alg:BFoRB}}.
				
				
				To that end, first we need to reformulate the problem.  In the formulation given in \cite[Section 2]{sagratella17}, the decision variable of each agent {(discrete flow, denoted by $a_i$)} is only integer and its mixed-strategy extension does not have a monotone pseudogradient because each agent is nonlinearly coupled in its cost function.  As noted in Remark \ref{rem:cost_trick}, we can overcome this issue by reformulating the game as a mixed-integer one, as follows:
				\begin{equation}	
					\hspace{-3pt} \begin{rcases}
						\underset{(a_i, y_i) \in \bb Z \times \bb R}{\min} \  & J_i^{\rm c}(y_i,\bm y_{-i}) +  J_i^{\rm d}(a_i)  \\
						\hspace{18pt} 	\operatorname{s.t.} & a_i - y_i = 0, \quad 
						0 \leq a_i \leq \bar{a}_i, \\
						& 0 \leq y_i \leq \bar{a}_i, \quad  
						\sum_{j \in \mc I} H_j^{\rm d} a_j \leq \rho,
					\end{rcases}
					\forall i \in \mc I,
					\label{eq:dfc_game}
				\end{equation}
				where, {similarly to \cite[Section 2]{sagratella17},} the cost function in {\eqref{eq:dfc_game}} is defined by
				{$	J_i^{\rm c}(y_i,\bm y_{-i}) \hspace{-2pt}= \sum_{l \in \mc P_i}\frac{q_l}{b_l + \rho_l - \sum_{j: l \in \mc P_j}y_j}$ and 
					$J_i^{\rm d}(a_i) = d_i \operatorname{log}(e_i(1+a_i)).$
				Here,} $\mc P_i$ denotes {a predefined} set of links used by agent $i$, $\rho_l$ denotes the capacity of link $l$, whereas $q_l$, $b_l$, $d_i$, and $e_i$ are positive constants. Note that the partial pseudogradient mapping $F^c(\bm y)$ is monotone. 
			Moreover, the vector $H_i^{\rm d} \in \bb R^L$, for each $i\in \mc I$, defines which links are used by player $i$, i.e., $[H_i^{\rm d}]_l = 1$ means that link $l$ is used by player $i$, and $[H_i^{\rm d}]_l = 0$ otherwise. {Additionally, we define}  $\rho = \col((\rho_l)_{l\in \mc E})$. We also have an upper bound on $a_i$, denoted by $\bar{a}_i$. Then, we have that $\mc A_i =\{a \in \bb Z \mid 0 \leq a \leq \bar{a}_i \}$ and, similarly, $\mc Y_i =\{y \in \bb R \mid 0 \leq y \leq \bar{a}_i \}$. 
			Therefore, we can obtain the mixed-strategy extension of \eqref{eq:dfc_game} by following the procedure in Section \ref{sec:gmi_games} and compute a variational MS-GNE of this game using one of the algorithms described in Section \ref{sec:GNE_alg}. Finally, an initial point for the GS algorithm  is obtained based on the variational MS-GNE that we have computed, e.g., by taking the integer decisions that have the highest probability. 
			
			\begin{figure}[t]
				\centering
				\includegraphics[scale=0.55]{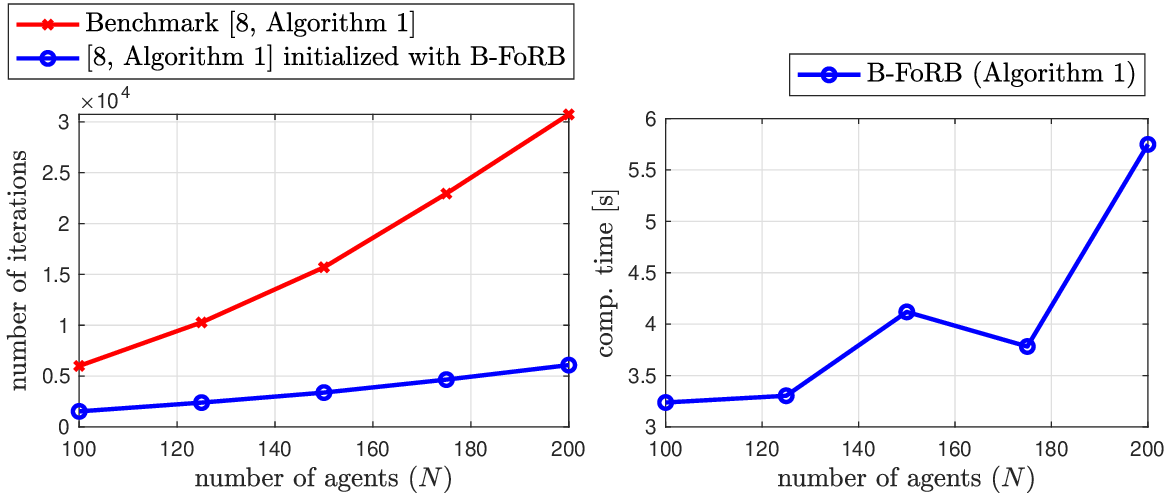}
				\caption{The {left} plot shows the comparison of the standard Gauss-Southwell algorithm \cite{sagratella17} (benchmark) and the modified one, which is initialized with the B-FoRB (Algorithm \ref{alg:BFoRB}), and the {right} plot shows the computational time/agent of the B-FoRB algorithm to compute an MS-GNE. Each data point is the average of {100} Monte-Carlo simulations.}
				\label{fig:gs_comp}
			\end{figure}
			
			We test the application of Algorithm \ref{alg:BFoRB} as a heuristic method for the initialization of the GS algorithm \cite[Algorithm 1]{sagratella17} via numerical experiments. We vary the number of agents $N =\{100,125,\dots,200\}$ and set the number of links to $60$. For each $n$, we generate 100 Monte-Carlo simulations, in which the cost coefficients follow \cite[Section 5.3]{sagratella17} whereas the links and the parameters in the constraints are randomly generated. Then, we compare the benchmark GS algorithm, which uses zero initial point, i.e., $\bm a^{(0)} = 0$, as specified in \cite[Section 5.3]{sagratella17}, and the GS algorithm initialized based on the solution of the B-FoRB algorithm (Algorithm \ref{alg:BFoRB}) {as previously explained}. Note that we set the approximation parameter $\varepsilon = 10^{-6}$ for the GS algorithm. Figure \ref{fig:gs_comp} shows the effectiveness of Algorithm \ref{alg:BFoRB} as an initialization heuristic as we observe a significant improvement in terms of the number of iterations of the GS algorithm (top plot). The price that we have to pay to reduce the number of iterations of the GS algorithm is reflected by the computational time/agent of the B-FoRB algorithm (shown in the bottom plot of Figure \ref{fig:gs_comp}). This extra computational time is relatively low compared to the total computational time of the GS method, { which} has to solve a mixed-integer programming at each iteration.

				\section{Conclusion}
				{The problem of finding a mixed-strategy generalized Nash equilibrium (MS-GNE) of  mixed-integer games with monotone pseudogradient can be formulated as an inclusion problem with a special structure. Based on this formulation, {a  novel Bregman forward-reflected-backward splitting method allows us to derive computationally inexpensive equilibrium seeking algorithms.} From our numerical experiments, the proposed algorithms are efficient not only to compute an MS-GNE but also as heuristic methods that initialize a pure GNE seeking algorithm for some classes of potential mixed-integer games.} 
				Computing an MS-GNE under time-varying and asynchronous communication networks are currently open research challenges. 
				Furthermore, recovering a pure (or an approximate pure) GNE from an MS-GNE is also an open problem. \color{black}
				\appendix 
				
				\section{Proof of Theorem \ref{th:conv_BreFoRB}}
				\label{ap:pf:th:conv_BreFoRB}
				
				
				
				\textit{Well-defined iterates:}  {Since $A$ is maximally monotone, $\Gamma A$ is maximally monotone in the $\Gamma^{-1}$-induced norm \cite[Proposition 20.24]{bauschke11}. 
					Furthermore, since $\varphi$ is a strictly convex function on $\operatorname{int}\dom(\varphi)$  {and separable (Assumption \ref{as:phi}), while} $\Gamma$ is a positive diagonal matrix, it holds that $\langle {\Gamma^{-1}}\nabla \varphi(\bm \omega)- {\Gamma^{-1}}\nabla \varphi(\bm \omega'), \bm \omega - \bm \omega' \rangle = \langle \nabla \varphi(\bm \omega)- \nabla \varphi(\bm \omega'), \bm \omega - \bm \omega' \rangle_{\Gamma^{-1}} > 0$, for all $\bm \omega \neq \bm \omega'$ and $\bm \omega, \bm \omega' \in\operatorname{int}\dom(\varphi)$ i.e. $\nabla \varphi$ is strictly monotone in the $\Gamma^{-1}$-induced norm. Hence, $(\nabla \varphi + \Gamma A)$ is strictly monotone in the $\Gamma^{-1}$-induced norm.} {Now, we suppose that $(z,{\bm{\omega}})$ and $(z,{\bm{\omega}}')$ belong to $\gph(\nabla \varphi + \Gamma A)^{-1}$. Then $z \in (\nabla \varphi + \Gamma A)({\bm{\omega}})$ and $z \in (\nabla \varphi + \Gamma A)({\bm{\omega}}')$. 
					Since $\langle z - z, \bm \omega -\bm \omega' \rangle_{\Gamma^{-1}} = 0$ and  $(\nabla \varphi + \Gamma A)$ is strictly monotone in the $\Gamma^{-1}$-induced norm, it must hold that $\bm \omega =\bm \omega'$ \cite[Def. 22.1.(ii)]{bauschke11}.  
				} Thus, we conclude that $(\nabla \varphi + \Gamma A)^{-1}$ is single-valued on $\operatorname{ran} (\nabla \varphi + \Gamma A)$. Moreover, since $\varphi$ is {essentially smooth, i.e. it is} continuously differentiable on $\operatorname{int}\dom(\varphi)$,  $\operatorname{ran} (\nabla \varphi + \Gamma A)^{-1}= \dom (\nabla \varphi) \cap \dom (A) = \operatorname{int} \dom (\varphi) \cap \dom (A)$ {\cite[Sec. 1.2]{bauschke11}}.  Furthermore, since $B$ is Lipschitz, $\nabla \varphi({\bm{\omega}}^{(0)}) - \Gamma (2B({\bm{\omega}}^{(0)})-B({\bm{\omega}}^{(-1)}))$ is singleton. We deduce ${\bm{\omega}}_1 \in \operatorname{int} \dom (\varphi) \cap \dom( A)$ is uniquely defined. Therefore, by induction we have that $({\bm{\omega}}^{(k)})_{k \in \bb N} \in \operatorname{int} \dom (\varphi) \cap \dom (A)$.

				{The remaining part of the proof builds upon \cite[Thm 2.5]{malitsky20}}. Specifically, by \cite[Lemma 2.46]{bauschke11}, we need to show the boundedness of the sequence $({\bm{\omega}}^{(k)})_{k \in \bb N}$ and the uniqueness of its cluster point {to prove convergence}. \color{black}
				
				\textit{Boundedness of iterates :}  By premultiplying both sides of the equality in \eqref{eq:BreFoRB} with $(\nabla \varphi + \Gamma A) $ and then rearranging the terms, we have that
				\begin{align}
					0 &\in \nabla \varphi ({\bm{\omega}}^{(k+1)}) + \Gamma A({\bm{\omega}}^{(k+1)}) - \nabla \varphi ({\bm{\omega}}^{(k)}) 
					+ \Gamma (2B({\bm{\omega}}^{(k)})-B({\bm{\omega}}^{(k-1)})). \label{eq:th1_incl1} 
				\end{align}
				Then, we multiply by $\Gamma^{-1}$ and rearrange the inclusion to obtain that
				{\small \begin{align}
						({\bm{\omega}}^{(k+1)}, \Gamma^{-1}  \nabla \varphi ({\bm{\omega}}^{(k)})  - (2B({\bm{\omega}}^{(k)})-B({\bm{\omega}}^{(k-1)})) 
						-  \Gamma^{-1}  \nabla \varphi({\bm{\omega}}^{(k+1)})) \in \gph( A). \label{eq:th1_incl1a}
				\end{align}}
				Moreover, consider ${\bm{\omega}} \in \zer(A + B) \cap \operatorname{int} \dom (\varphi)$, which implies $({\bm{\omega}},-B({\bm{\omega}})) \in \gph (A)$ and recall that $({\bm{\omega}}^{(k)})_{k \in \bb N} \in  \dom (A)$. By monotonicity of $A$ as well as considering the pairs $({\bm{\omega}},-B({\bm{\omega}}))$ and \eqref{eq:th1_incl1a}, we have that
				\begin{align}
					0 & \leq 
					\langle \nabla \varphi({\bm{\omega}}^{(k+1)})-\nabla \varphi({\bm{\omega}}^{(k)}), {\bm{\omega}} - {\bm{\omega}}^{(k+1)} \rangle \notag_{\Gamma^{-1}} 
					+  \langle B({\bm{\omega}}^{(k)}) -B({\bm{\omega}}^{(k-1)}), {\bm{\omega}} - {\bm{\omega}}^{(k+1)}\rangle 
					\notag\\
					& \quad 
					+  \langle B({\bm{\omega}}^{(k)}) -B({\bm{\omega}}), {\bm{\omega}} - {\bm{\omega}}^{(k+1)}\rangle \notag.
				\end{align}
				Notice that $\Gamma^{-1}:=\diag((\gamma_i^{-1})_{i=1}^{{n_\omega}})$ is a positive diagonal matrix.  By denoting the $i$-th component of $\bm \omega$ as $\omega_i \in \bb R$, we can write the Legendre function $\varphi(\bm \omega) = \sum_{i=1}^{{n_\omega}} \varphi_i(\omega_i)$. Therefore, we can instead consider the function $\hat{\varphi}(\bm \omega)=\sum_{i =1}^{{n_\omega}} \gamma_i^{-1} \varphi_i(\omega_i)$, which is also a  {Legendre} function, with $\dom(\hat \varphi) = \dom(\varphi)$  {(c.f. the proof of \cite[Theorem 5.12]{bauschke97})}, and has gradient $\nabla \hat{\varphi}(\bm \omega) = \Gamma^{-1} \nabla \varphi (\bm \omega)$. Hence, by considering $\hat{\varphi}$ \color{black} and applying the three-point identity of Bregman distance \cite[Proposition 2.3.(ii)]{bauschke03}, the first addend on the right-hand side of the inequality can be written as  	
					$\langle \nabla \varphi({\bm{\omega}}^{(k+1)})-\nabla \varphi({\bm{\omega}}^{(k)}), {\bm{\omega}} - {\bm{\omega}}^{(k+1)} \rangle_{\Gamma^{-1}} 
					= \langle \nabla \hat \varphi({\bm{\omega}}^{(k+1)})-\nabla \hat \varphi({\bm{\omega}}^{(k)}), {\bm{\omega}} - {\bm{\omega}}^{(k+1)} \rangle
					= \dist_{\hat \varphi}({\bm{\omega}},{\bm{\omega}}^{(k)}) - \dist_{\hat \varphi}({\bm{\omega}},{\bm{\omega}}^{(k+1)}) 
					- \dist_{\hat \varphi}({\bm{\omega}}^{(k+1)},{\bm{\omega}}^{(k)}).
					$ 
					Moreover, we rewrite the last two addends as 
						$
						\langle B({\bm{\omega}}^{(k)}) -B({\bm{\omega}}^{(k-1)}), {\bm{\omega}} - {\bm{\omega}}^{(k+1)}\rangle 
						\hspace{1.5pt}=\hspace{1.5pt} \langle B({\bm{\omega}}^{(k)}) -B({\bm{\omega}}^{(k-1)})), {\bm{\omega}} - {\bm{\omega}}^{(k)}\rangle
						+ \langle B({\bm{\omega}}^{(k)}) -B({\bm{\omega}}^{(k-1)}), {\bm{\omega}}^{(k)} - {\bm{\omega}}^{(k+1)}\rangle;$ 
						and 
						$
						\langle  B({\bm{\omega}}^{(k)}) \hspace{-0.8pt}-\hspace{-0.8pt}B({\bm{\omega}}), {\bm{\omega}}\hspace{-0.8pt}-\hspace{-0.8pt} {\bm{\omega}}^{(k+1)}\rangle 
						\hspace{-1.8pt}=\hspace{-1.8pt}\langle B({\bm{\omega}}^{(k)}) \hspace{-0.8pt}-\hspace{-0.8pt}B({\bm{\omega}}^{(k+1)}), {\bm{\omega}} - {\bm{\omega}}^{(k+1)}\rangle 
						+  \langle B({\bm{\omega}}^{(k+1)}) -B({\bm{\omega}}), {\bm{\omega}} - {\bm{\omega}}^{(k+1)}\rangle.
						$
						Thus, we obtain
						\begin{equation*}
							\begin{aligned}
								0 &\leq  \dist_{\hat \varphi}({\bm{\omega}},{\bm{\omega}}^{(k)}) - \dist_{\hat \varphi}({\bm{\omega}},{\bm{\omega}}^{(k+1)}) 
								- \dist_{\hat \varphi}({\bm{\omega}}^{(k+1)},{\bm{\omega}}^{(k)}) 
								\\
								&\quad 
								+   \langle B({\bm{\omega}}^{(k)}) -B({\bm{\omega}}^{(k-1)})), {\bm{\omega}} - {\bm{\omega}}^{(k)}\rangle 
								+ \langle B({\bm{\omega}}^{(k)}) -B({\bm{\omega}}^{(k-1)}), {\bm{\omega}}^{(k)} - {\bm{\omega}}^{(k+1)}\rangle 
								\\
								&\quad + \langle B({\bm{\omega}}^{(k)}) -B({\bm{\omega}}^{(k+1)}), {\bm{\omega}} - {\bm{\omega}}^{(k+1)}\rangle  
								+  \langle B({\bm{\omega}}^{(k+1)}) -B({\bm{\omega}}), {\bm{\omega}} - {\bm{\omega}}^{(k+1)}\rangle.
							\end{aligned}
						\end{equation*}

						We can rearrange and upper bound {the preceding inequality}  as follows:
						\begin{align}
							\label{eq:th1_ineq2}
							&\dist_{\hat \varphi}({\bm{\omega}},{\bm{\omega}}^{(k+1)}) +\dist_{\hat \varphi}({\bm{\omega}}^{(k+1)},{\bm{\omega}}^{(k)})  
							+ \langle B({\bm{\omega}}^{(k+1)})- B({\bm{\omega}}^{(k)}), {\bm{\omega}} - {\bm{\omega}}^{(k+1)}\rangle 
							\\
							&\leq \dist_{\hat \varphi}({\bm{\omega}},{\bm{\omega}}^{(k)})  + \langle B({\bm{\omega}}^{(k)}) -B({\bm{\omega}}^{(k-1)})), {\bm{\omega}} - {\bm{\omega}}^{(k)}\rangle  \notag\\
							&\quad + \langle B({\bm{\omega}}^{(k)}) -B({\bm{\omega}}^{(k-1)}), {\bm{\omega}}^{(k)} - {\bm{\omega}}^{(k+1)}\rangle 
							+   \langle B({\bm{\omega}}^{(k+1)}) -B({\bm{\omega}}), {\bm{\omega}} - {\bm{\omega}}^{(k+1)}\rangle 
							\notag \\
							&\leq \dist_{\hat \varphi}({\bm{\omega}},{\bm{\omega}}^{(k)})  + \langle B({\bm{\omega}}^{(k)}) -B({\bm{\omega}}^{(k-1)})), {\bm{\omega}} - {\bm{\omega}}^{(k)}\rangle 
							\notag \\
							&\quad 
							+ \frac{ \ell_B}{2} \left(\|{{\bm{\omega}}^{(k)}-{\bm{\omega}}^{(k-1)}}\|^2 + \|   {\bm{\omega}}^{(k+1)}\hspace{-2pt}-\hspace{-2pt}{\bm{\omega}}^{(k)}\|^2 \right), 	\notag
						\end{align}
						where the second inequality is obtained by discarding $ \langle B({\bm{\omega}}^{(k+1)}) -B({\bm{\omega}}), {\bm{\omega}} - {\bm{\omega}}^{(k+1)}\rangle$, which is nonpositive  due to the monotonicity of $B$ and 
						considering an upper-bound of the term $\langle B({\bm{\omega}}^{(k)}) -B({\bm{\omega}}^{(k-1)}), {\bm{\omega}}^{(k)} - {\bm{\omega}}^{(k+1)}\rangle$ due to the Lipschitz continuity of $B$, i.e., 
						$\langle B({\bm{\omega}}^{(k)}) -B({\bm{\omega}}^{(k-1)}), {\bm{\omega}}^{(k)} - {\bm{\omega}}^{(k+1)}\rangle
						\leq  \frac{ \ell_B}{2}\hspace{-3pt} \left(\|{{\bm{\omega}}^{(k)}\hspace{-2pt}-\hspace{-2pt}{\bm{\omega}}^{(k-1)}}\|^2 \hspace{-2pt}+\hspace{-2pt} \|   {\bm{\omega}}^{(k+1)}\hspace{-2pt}-\hspace{-2pt}{\bm{\omega}}^{(k)}\|^2 \right).$ \color{black}
						Since $\varphi$ is $\sigma_{\varphi}$-strongly convex {on $\operatorname{int}\dom(\varphi)\cap\dom(A)$}, we observe that $\hat \varphi$ is $\sigma_{\hat \varphi}$-strongly convex {on $\operatorname{int}\dom(\varphi)\cap\dom(A)$}, where $\sigma_{\hat \varphi} = \sigma_{\varphi}/(\max_{i} \gamma_i)$. \color{black} Therefore, {since $\bm \omega^{(k)} \in \operatorname{int}\dom(\varphi)\cap\dom(A)$, for all $k \in \bb N$,} we also have that
						$\frac{\sigma_{\hat \varphi}}{2} \|{\bm{\omega}}^{(k+1)}-{\bm{\omega}}^{(k)}\|^2 \leq \dist_{\hat \varphi}({\bm{\omega}}^{(k+1)},{\bm{\omega}}^{(k)})$, {for all $k \in \bb N$}, which, {together with  \eqref{eq:th1_ineq2},} results in
						\begin{equation}
							\begin{aligned}
								&\dist_{\hat \varphi}({\bm{\omega}},{\bm{\omega}}^{(k+1)}) + \langle B({\bm{\omega}}^{(k+1)})- B({\bm{\omega}}^{(k)}), {\bm{\omega}} - {\bm{\omega}}^{(k+1)}\rangle 
								+ \frac{\sigma_{\hat \varphi} -  \ell_B }{2} \|{\bm{\omega}}^{(k+1)}-{\bm{\omega}}^{(k)}\|^2  
								\\
								&\leq \dist_{\hat \varphi}({\bm{\omega}},{\bm{\omega}}^{(k)})  + \langle B({\bm{\omega}}^{(k)}) -B({\bm{\omega}}^{(k-1)})), {\bm{\omega}} - {\bm{\omega}}^{(k)}\rangle 
								+ \frac{\ell_B }{2} \|{\bm{\omega}}^{(k)}-{\bm{\omega}}^{(k-1)}\|^2. 
							\end{aligned}
							\label{eq:th1_ineq3}
						\end{equation}
						
						Since $\gamma_i \in (0, \frac{\sigma_{ \varphi}}{2 \ell_B} )$, for all $i=1,\dots,n$, it holds that $ \ell_B (\max_i \gamma_i) < \frac{\sigma_{ \varphi}}{2} \Leftrightarrow \ell_B < \frac{\sigma_{ \hat \varphi}}{2} $ and there exists $\varepsilon > 0$ such that $\max_i \gamma_i = \frac{\sigma_{ \varphi}-2\varepsilon}{2 \ell_B}$, implying $\sigma_{ \hat{\varphi}} - \ell_B  = \frac{\sigma_{ \hat \varphi}}{2} + \varepsilon $. \color{black} Combining these facts and \eqref{eq:th1_ineq3}, we obtain that
						\begin{equation}
							\begin{aligned}
								&\langle B({\bm{\omega}}^{(k+1)})- B({\bm{\omega}}^{(k)}), {\bm{\omega}} - {\bm{\omega}}^{(k+1)}\rangle 
								+\dist_{\hat \varphi}({\bm{\omega}},{\bm{\omega}}^{(k+1)}) + \tfrac{\sigma_{\hat \varphi} +2 \varepsilon}{4} \|{\bm{\omega}}^{(k+1)}-{\bm{\omega}}^{(k)}\|^2  
								\\
								&\leq  \langle B({\bm{\omega}}^{(k)}) -B({\bm{\omega}}^{(k-1)})), {\bm{\omega}} - {\bm{\omega}}^{(k)}\rangle 
								+ \dist_{\hat \varphi}({\bm{\omega}},{\bm{\omega}}^{(k)}) + \tfrac{\sigma_{\hat \varphi}}{4} \|{\bm{\omega}}^{(k)}-{\bm{\omega}}^{(k-1)}\|^2. 
							\end{aligned}
							\label{eq:th1_ineq4}
						\end{equation}
						
						By summing \eqref{eq:th1_ineq4}
						over the iterates until $k$, we obtain that
						\begin{align}
							&\sum_{l=0}^k \Big( \langle B({\bm{\omega}}^{(l+1)})- B({\bm{\omega}}^{(l)}), {\bm{\omega}} - {\bm{\omega}}^{(l+1)}\rangle 
							\hspace{-1.5pt}+\hspace{-1.5pt}\dist_{\hat \varphi}({\bm{\omega}},{\bm{\omega}}^{(l+1)}) + \tfrac{\sigma_{\hat  \varphi} \hspace{-1.5pt}+\hspace{-1.5pt}2 \varepsilon}{4} \|{\bm{\omega}}^{(l+1)}-{\bm{\omega}}^{(l)}\|^2 \Big)
							\notag \\
							&\leq \sum_{l=0}^k \Big(\Big. \langle B({\bm{\omega}}^{(l)}) -B({\bm{\omega}}^{(l-1)})), {\bm{\omega}} - {\bm{\omega}}^{(l)}\rangle 
							+ \dist_{\hat \varphi}({\bm{\omega}},{\bm{\omega}}^{(l)}) + \tfrac{\sigma_{\hat \varphi}}{4} \|{\bm{\omega}}^{(l)}-{\bm{\omega}}^{(l-1)}\|^2 \Big.\Big), \notag
						\end{align}
						which consists of a telescoping sum and is equivalent to
						\begin{equation}
							\begin{aligned}
								&\langle B({\bm{\omega}}^{(k+1)})- B({\bm{\omega}}^{(k)}), {\bm{\omega}} - {\bm{\omega}}^{(k+1)}\rangle 
								+\dist_{\hat \varphi}({\bm{\omega}},{\bm{\omega}}^{(k+1)}) + \tfrac{\sigma_{\hat \varphi} }{4} \|{\bm{\omega}}^{(k+1)}-{\bm{\omega}}^{(k)}\|^2 
								\\
								&
								+ \tfrac{ \varepsilon}{2}\sum_{l=0}^k  \|{\bm{\omega}}^{(l+1)}-{\bm{\omega}}^{(l)}\|^2  
								\\
								&
								\leq  \langle B({\bm{\omega}}^{(0)}) -B({\bm{\omega}}^{(-1)})), {\bm{\omega}} - {\bm{\omega}}^{(0)}\rangle 
								+\dist_{\hat \varphi}({\bm{\omega}},{\bm{\omega}}^{(0)}) + \tfrac{\sigma_{\hat \varphi}}{4} \|{\bm{\omega}}^{(0)}-{\bm{\omega}}^{(-1)}\|^2.
							\end{aligned}
							\label{eq:th1_ineq5}
						\end{equation}
						
						Now, we lower bound the term $\langle B({\bm{\omega}}^{(k+1)})- B({\bm{\omega}}^{(k)}), {\bm{\omega}} - {\bm{\omega}}^{(k+1)}\rangle$ via the Lipschitz continuity of $B$, i.e.,
						\begin{align}
							\langle B({\bm{\omega}}^{(k+1)})- B({\bm{\omega}}^{(k)}), {\bm{\omega}} - {\bm{\omega}}^{(k+1)}\rangle 
							&\geq - \frac{ \ell_B }{2} \left( \|{\bm{\omega}}^{(k+1)}-{\bm{\omega}}^{(k)}\|^2 +\|{\bm{\omega}}^{(k+1)}-{\bm{\omega}} \|^2 \right) \notag\\
							&\geq - \frac{\sigma_{\hat \varphi}}{4} \|{\bm{\omega}}^{(k+1)}-{\bm{\omega}}^{(k)}\|^2 - \frac{ \ell_B }{2}\|{\bm{\omega}}^{(k+1)}-{\bm{\omega}} \|^2,
							\label{eq:th1_ineq6}
						\end{align} \color{black}
						where the last inequality holds due to the choice of $\Gamma$, i.e., $\max_i \gamma_i \ell_B < \frac{\sigma_{\varphi}}{2}$. 
						By combining \eqref{eq:th1_ineq5} and \eqref{eq:th1_ineq6} as well as using the fact that $\frac{\sigma_{\hat \varphi}}{2} \|{\bm{\omega}}^{(k+1)} - {\bm{\omega}} \|^2 \leq \dist_{\hat \varphi}({\bm{\omega}},{\bm{\omega}}^{(k+1)})$ and ${\bm{\omega}}^{(0)} = {\bm{\omega}}^{(-1)}$, we have that
						\begin{align}
							\frac{\sigma_{\hat \varphi}\hspace{-2pt}-\hspace{-2pt}\ell_B }{2} \|{\bm{\omega}}^{(k+1)}\hspace{-2pt}-\hspace{-2pt}{\bm{\omega}} \|^2 \hspace{-2pt}+\hspace{-2pt} \frac{ \varepsilon}{2}\sum_{l=0}^k  \|{\bm{\omega}}^{(l+1)}\hspace{-2pt}-\hspace{-2pt}{\bm{\omega}}^{(l)}\|^2  
							\leq \dist_{\hat \varphi}({\bm{\omega}},{\bm{\omega}}^{(0)}) , \label{eq:th1_ineq7}
						\end{align}
						where $\sigma_{\hat \varphi}-\ell_B> 0$. The inequality \eqref{eq:th1_ineq7} implies that $\lim_{k\to \infty} \|{\bm{\omega}}^{(k+1)} - {\bm{\omega}}^{(k)} \|^2 = 0$. Moreover, since $\frac{\sigma_{\hat \varphi}- \ell_B }{2} \|{\bm{\omega}}^{(k+1)}-{\bm{\omega}} \|^2 \leq \dist_{\hat \varphi}({\bm{\omega}},{\bm{\omega}}^{(0)}) < +\infty$, for all $k \in \bb N$, we deduce that the sequence $({\bm{\omega}}^{(k)})_{k \in \bb N}$ is bounded. 
						
						
						
						\textit{Cluster points are solutions:} 
						Let $\hat {\bm{\omega}}$ be an arbitrary  cluster point of the bounded sequence $({\bm{\omega}}^{(k)} )_{k\in\bb N}$. Now, we  show that $\hat {\bm{\omega}}$ is contained in  $\zer(A+B) \cap \operatorname{int} \dom ( \varphi)$.  { To this end, first we show that, for any fixed point $\bm z \in \zer(A+B) \cap \operatorname{int}( \varphi)$ (which is assumed to exist), $\dist_{\hat \varphi}(\bm z, \bm{\omega}^{(k)})$ converges.  By combining \eqref{eq:th1_ineq4} and \eqref{eq:th1_ineq6}, we deduce that the right-hand side of the inequality \eqref{eq:th1_ineq4} is non-negative for all $k \geq 0$. Therefore, the inequality \eqref{eq:th1_ineq4} and \cite[Lemma 5.31]{bauschke11} assert the existence of the limit 
							\begin{align}
								\lim_{k \to \infty} \Big(\Big. \dist_{\hat \varphi}(\bm z,{\bm{\omega}}^{(k)}) + \frac{\sigma_{\hat \varphi}}{4} \|{\bm{\omega}}^{(k)}-{\bm{\omega}}^{(k-1)}\|^2 
								+ \langle B({\bm{\omega}}^{(k)}) -B({\bm{\omega}}^{(k-1)})), \bm z - {\bm{\omega}}^{(k)}\rangle \Big.\Big). 
								\label{eq:alpha_seq}
							\end{align}
							Since $B$ is continuous and $\lim_{k \to \infty}\|{\bm{\omega}}^{(k)}-{\bm{\omega}}^{(k-1)}\|^2=0$, the limit in \eqref{eq:alpha_seq} is equal to the limit of $\dist_{\hat \varphi}(\bm z, \bm{\omega}^{(k)})$, implying that the sequence $(\dist_{\hat \varphi}(\bm z, \bm{\omega}^{(k)}))_{k \in \bb N}$ is bounded.  Since $\varphi$ is essentially smooth, $\bm z \in \operatorname{int}\dom( \varphi)$, $(\bm \omega^{(k)})_{k \in \bb N} \in  \operatorname{int}\dom( \varphi)$, and the sequence $(\dist_{\hat \varphi}(\bm z, \bm{\omega}^{(k)}))_{k \in \bb N}$ is bounded, \cite[Theorem 3.8(ii)]{bauschke97} asserts that every cluster point belongs to $\operatorname{int}\dom(\varphi)$.} Furthermore, since we assume that $\zer(A+B) \cap \operatorname{int} \dom ( \varphi) \neq \varnothing$, we only need to show that $0 \in (A+B)(\hat {\bm{\omega}})$. From \eqref{eq:th1_incl1}, we obtain
						{\small 	\begin{align}
								&\Gamma^{-1}\left(\nabla  \varphi ({\bm{\omega}}^{(k)})-\nabla  \varphi ({\bm{\omega}}^{(k+1)})\right) -2B({\bm{\omega}}^{(k)}) 
								+B({\bm{\omega}}^{(k-1)}) + B({\bm{\omega}}^{(k+1)}) \in   (A+B)({\bm{\omega}}^{(k+1)}),   \label{eq:th1_incl2}
						\end{align}}
						for all $k \geq 0$. 
						Since  $(A+B)$ is maximally monotone, $\gph(A+B)$ is sequentially closed  \cite[Proposition 20.38]{bauschke11}. Additionally, $B$ is continuous and $\nabla \hat \varphi(\cdot)$ is also continuous in $\operatorname{int} \dom ( \varphi)$ (Definition \ref{def:Breg_f}). Therefore, by taking the limit along a subsequence of $( {\bm{\omega}}^{(k)})$ that converges to $\hat {\bm{\omega}}$ in \eqref{eq:th1_incl2},  the limit of left-hand side term of the inclusion \eqref{eq:th1_incl2} is 0, and indeed $0 \in (A+B)(\hat {\bm{\omega}})$.  

						\textit{Uniqueness of  cluster point:}  It remains to show that the set of  cluster points of $({\bm{\omega}}^{(k)} )_{k\in\bb N}$ is a singleton.  
						{We recall that for an arbitrary cluster point $\hat {\bm \omega} \in \zer(A+B) \cap \operatorname{int}\dom(\varphi)$,  $\lim_{k\to \infty} \dist_{\hat \varphi}(\hat{\bm \omega}, \bm{\omega}^{(k)})$ exists.} Now, we take two subsequences $\bm{\omega}^{(k_n)}$ and $\bm{\omega}^{(l_n)}$ and denote their  cluster points by $ \hat{\bm \omega}_1$ and $\hat{\bm \omega}_2$, respectively.  We know that the sequences $(\dist_{\hat \varphi}(\hat{\bm \omega}_1, \bm{\omega}^{(k)}))_{k \in \bb N}$ and $(\dist_{\hat \varphi}(\hat{\bm \omega}_2, \bm{\omega}^{(k)}))_{k \in \bb N}$ converge. Therefore, using the three-point identity of Bregman distance \cite[Proposition 2.3.(ii)]{bauschke03}, we have that the sequence
						$(\langle \hat{\bm \omega}_1 - \hat{\bm \omega}_2, \nabla \hat \varphi(\bm{\omega}^{(k)}) - \nabla \hat \varphi(\hat{\bm \omega}_2) \rangle)_{k\in \bb N}$,  
					which is equal to
					$(\dist_{\hat \varphi}(\hat{\bm \omega}_2, \bm{\omega}^{(k)})+\dist_{\hat \varphi}(\hat{\bm \omega}_1, \hat{\bm \omega}_2)-\dist_{\hat \varphi}(\hat{\bm \omega}_1, \bm{\omega}^{(k)}))_{k \in \bb N}$
					converges. Taking the subsequence $\bm{\omega}^{(l_n)}$ and using the continuity of $\hat \varphi$, we obtain that 
					$\eta 
					= \lim_{n \to \infty} \langle \hat{\bm \omega}_1 - \hat{\bm \omega}_2, \nabla \hat \varphi(\bm{\omega}^{(l_n)}) - \nabla \hat \varphi(\hat{\bm \omega}_2) \rangle 
					= \langle \hat{\bm \omega}_1 - \hat{\bm \omega}_2, \nabla \hat \varphi(\hat{\bm{\omega}}_2) - \nabla \hat \varphi(\hat{\bm \omega}_2) \rangle = 0.
					$ 
					By the continuity of $\hat \varphi$ and considering the subsequence $\bm{\omega}^{(k_n)}$, we also have that 
					$
					\lim_{n \to \infty} \langle \hat{\bm \omega}_1 - \hat{\bm \omega}_2, \nabla \hat \varphi(\bm{\omega}^{(k_n)}) - \nabla \hat \varphi(\hat{\bm \omega}_2) \rangle = \eta.
					$
					Since $\eta = 0$, we obtain that
					$
					\lim_{n \to \infty} \langle \hat{\bm \omega}_1 - \hat{\bm \omega}_2, \nabla \hat \varphi(\bm{\omega}^{(k_n)}) - \nabla \hat \varphi(\hat{\bm \omega}_2) \rangle 
					= \langle \hat{\bm \omega}_1 - \hat{\bm \omega}_2, \nabla \hat \varphi(\hat{\bm{\omega}}_1) - \nabla \hat \varphi(\hat{\bm \omega}_2) \rangle = 0.
					$ 
					However, $\hat \varphi$ is {essentially strictly} convex. Hence, the preceding equality holds only if $\hat{\bm{\omega}}_1 =\hat{\bm \omega}_2$. We conclude that the cluster point of $({\bm{\omega}}^{(k)} )_{k\in\bb N}$ is unique. \qedd
					
					\section{Proof of Lemma \ref{le:Lips_B}}
					\label{ap:pf:le:Lips_B}
					Due to Assumption \ref{as:mon_F}, $\mc T_2$ is maximally monotone since it concatenates $F$, which is maximally monotone, and a constant vector $\begin{bmatrix} \bm \theta^{\top} & \rho^{\top} \end{bmatrix}^{\top}$. Furthermore, $\mc T_3$ is monotone  as can be observed that, for any {$\bm \omega_1,\bm \omega_2 \in \bb R^{m+n} \times \bb R_{\geq 0}^{n_\theta} \times \bb R^{ n_\rho} $}, 
					\begin{align*}
						\langle \mc T_3({\bm{\omega}}_1) - \mc T_3({\bm{\omega}}_2), {\bm{\omega}}_1 - {\bm{\omega}}_1 \rangle
						&= \langle g^{\rm c} (\bm y_2)+\nabla g^{\rm c}(\bm y_1)(\bm y_1-\bm y_2)-g^{\rm c} (\bm y_1),\bm \mu_1 \rangle\\
						&\quad + \langle g^{\rm c} (\bm y_1)-\nabla g^{\rm c}(\bm y_2)(\bm y_1-\bm y_2)-g^{\rm c} (\bm y_2),\bm \mu_2 \rangle
						\geq 0,
					\end{align*}%
					where the inequality holds since $g^{\rm c}$ is convex {and $\bm \mu_1,\bm \mu_2 \in \bb R_{\geq 0}^{n_\theta}$}. Moreover, $\mc T_3$ is maximally monotone since it is continuous. The maximal monotonicity of $\mc T_4$ can be proven similarly. Since $\dom (\mc T_2) = \bb R^{{n_{\bar \omega}}}$, the maximal monotonicity of $B$ follows from \cite[Corollary 25.5(i)]{bauschke11}. Finally, $\mc T_2$ is $\max(\ell_{F^{\rm d}},\ell_{F^{\rm c}})$-Lipschitz due to Assumption \ref{as:mon_F} whereas $\mc T_3$ and $\mc T_4$ are $\ell_{\mc T_3}$- and $\ell_{\mc T_4}$-Lipschitz, respectively, due to Assumption \ref{as:L_T34}. Thus, $B$ is also Lipschitz 
					with constant $\ell_B = \max_{j\in\{2,3,4\}}(\ell_{\mc T_j})$.
					\qedd
					

					\section{Proof of Lemma \ref{le:str_cvx_Bregf}}
					\label{ap:pf:le:str_cvx_Bregf}	
					The functions $\varphi_i^{\rm NE}\colon \bb R_{\geq 0}^{m_i} \to \bb R$, for all $i \in \mc I$, and  $\varphi^{\mathrm{E}}\colon \bb R^{n + n_\theta + n_\rho} \to \bb R$ are Legendre (see Examples \ref{ex:E_entropy}-\ref{ex:NE_entropy}). Therefore, the function $\varphi$ in \eqref{eq:Breg_f_gmi} is also Legendre (c.f. the proof of \cite[Theorem 5.12]{bauschke97}). 
					Furthermore, \cite[Proposition 5.1]{beck03} assesses that $\varphi_i^{\rm NE}$ is $1$-strongly convex on $\operatorname{rint}(\mc X_i)$ with respect to $1$-norm. Consequently, it is also 1-strongly convex with respect to the Euclidean norm since $\|\cdot \|_1 \geq \|\cdot\|$. Moreover,  $\varphi^{\mathrm{E}} \colon \bb R^{n+n_\theta+n_\rho} \to \bb R$ is  $1$-strongly convex on its domain with respect to the Euclidean norm. Thus, the function $\varphi$ in \eqref{eq:Breg_f_gmi} $1$-strongly convex on $\prod_{i\in \mc I}\operatorname{rint}(\mc X) \times \bb R^{n + n_\theta + n_\rho}$ with respect to the Euclidean norm. \qedd
					\color{black}
					
					\section{Proof of Theorem \ref{th:conv_alg_BFoRB}}
					\label{ap:pf:th:conv_alg_BFoRB}
					First, we show  that the updates in Algorithm \ref{alg:BFoRB} can be compactly written as \eqref{eq:BreFoRB} with $A$ and $B$ defined in \eqref{eq:op_split} and $\varphi$ defined in \eqref{eq:Breg_f_gmi} {in the next lemma}.
					
					\begin{lemma}
						\label{pr:BFoRB}
						The B-FoRB iteration \eqref{eq:BreFoRB} applied to the operators $A$ and  $B$ in \eqref{eq:op_split} with Legendre function $\varphi$ in \eqref{eq:Breg_f_gmi}  is equivalent to the updates in Algorithm \ref{alg:BFoRB}. 
					\end{lemma}
					\begin{proof}[Proof of Lemma \ref{pr:BFoRB}] Recall that B-FoRB  in \eqref{eq:BreFoRB} {can be equivalently written} as 
							$	0 \in \nabla \varphi({\bm{\omega}}^{(k+1)}) +  \Gamma A({\bm{\omega}}^{(k+1)} )
							-(\nabla \varphi ({\bm{\omega}}^{(k)}) - \Gamma (2B({\bm{\omega}}^{(k)})-B({\bm{\omega}}^{(k-1)})), \notag
							$
							implying that, {for  $A$ and  $B$ in \eqref{eq:op_split}, it holds that}
							{	
								\begin{align}
									{\bm{\omega}}^{(k+1)} 
									&= \underset{{\bm{\omega}} \in \mc X \times \mc Y \times \bb R^{n_\theta+n_\rho}_{\geq 0}}{\argmin} \Big( \dist_{\varphi}({\bm{\omega}},{\bm{\omega}}^{(k)}) \Big. 
									+\langle 2B({\bm{\omega}}^{(k)})-B({\bm{\omega}}^{(k-1)}), {\bm{\omega}} \rangle_\Gamma  \Big.\Big). \notag
							\end{align}}%
							Moreover, {$\varphi$ in \eqref{eq:Breg_f_gmi} as well as $A$ and $B$ in \eqref{eq:op_split}} are decomposable. Particularly, we can write 
							\begin{align*}
								B =&\col\left(\right.
								\col((B_i^x({\bm{\omega}}))_{i \in \mc I}),
								\col((B_i^y({\bm{\omega}}))_{i \in \mc I}), 
								\col((B_i^{\mu}({\bm{\omega}}))_{i \in \mc I}),
								B^{\lambda}({\bm{\omega}})\left.\right)
								,
							\end{align*} \color{black}
							where:
							\begin{align*}
								B_i^x({\bm{\omega}}) &= \nabla_{x_i} \bbi J_i^{\rm d} (\bm x) + G_i^{\rm d \top} \mu_i + H_i^{\rm d \top}\lambda;\ \ & B_i^{\mu}({\bm{\omega}}) = \theta_i - G_i^{\rm d} x_i - g_i^{\rm c}(y_i); \\
								B_i^y({\bm{\omega}}) &= \nabla_{y_i} J_i^{\rm c} (\bm y) + \nabla_{y_i}g_i^{\rm c}(\bm y)^{\top} \mu_i + \nabla_{y_i}h_i^{\rm c}(\bm y)^{\top}\lambda;\ \
								& B^{\lambda}({\bm{\omega}}) = \rho -H^{\rm d} \bm x - h^{\rm c}(\bm y).
							\end{align*}
							Let $\Gamma = \diag((\gamma_i I_{m_i})_{i \in \mc I},(\gamma_i I_{n_i})_{i \in \mc I},(\gamma_i I_{n_{\theta_i}})_{i \in \mc I}, \zeta I_{n_\rho})$,  which is also decomposable. Thus, we have that
							{\small	 \begin{align}		
									x_{i}^{(k+1)} &= \underset{x_i \in \Delta(\mc A_i)}{\argmin} \left( \gamma_i \langle 2B_i^x({\bm{\omega}}^{(k)}) - B_i^x({\bm{\omega}}^{(k-1)}), x_i \rangle \right. 
									\left. + \dist_{\varphi_i^{\mathrm{NE}}}(x_i,x_i^{(k)}) \right), \ \ &\forall i \in \mc I, \label{eq:x_upd}\\
									y_{i}^{(k+1)} &= \underset{y_i \in \mc Y_i}{\argmin} \left( \gamma_i \langle 2B_i^y({\bm{\omega}}^{(k)}) - B_i^y({\bm{\omega}}^{(k-1)}), y_i \rangle \right. 
									\left. + \dist_{\varphi_i^{\mathrm{E}}}(y_i,y_i^{(k)})\right), \ \ &\forall i \in \mc I, \label{eq:y_upd}	\\
									\mu_{i}^{(k+1)} &= \underset{\mu_i \in \bb R^{n_{\theta_i}}_{\geq 0}}{\argmin} \left( \gamma_i \langle 2B_i^{\mu}({\bm{\omega}}^{(k)}) - B_i^{\mu}({\bm{\omega}}^{(k-1)}), \mu_i \rangle \right. 
									\left. + \dist_{\varphi_i^{\mathrm{E}}}(\mu_i,\mu_i^{(k)})\right),\ \ &\forall i \in \mc I, \label{eq:mu_upd} \\
									\lambda^{(k+1)} &= \underset{\lambda \in \bb R^{n_\rho}_{\geq 0}}{\argmin} \ \left( \zeta \langle 2B^{\lambda}({\bm{\omega}}^{(k)}) - B^{\lambda}({\bm{\omega}}^{(k-1)}), \lambda \rangle \right. 
									\left. + \dist_{\varphi_i^{\mathrm{E}}}(\lambda,\lambda^{(k)})\right) \label{eq:lambda_upd}.
							\end{align}}
							The closed-form expression of \eqref{eq:x_upd} is given in \eqref{eq:x_upd_alg_BFoRB} \cite[Example 4.2]{belmega18}, whereas \eqref{eq:y_upd}, \eqref{eq:mu_upd}, and \eqref{eq:lambda_upd} are equivalently written as   \eqref{eq:y_upd_alg_BFoRB}, \eqref{eq:mu_upd_alg_BFoRB}, and \eqref{eq:lambda_upd_alg_BFoRB}. 
						\end{proof}

						Furthermore, by Assumptions \ref{as:cont}, \ref{as:mon_F}, and \ref{as:L_T34}, the operators $A$ and $B$ in  \eqref{eq:op_split} indeed satisfy the sufficient condition for the convergence of the B-FoRB {stated in Theorem \ref{th:conv_BreFoRB}}, i.e., both are maximally monotone  and $B$ is Lipschitz continuous.  {Moreover, $\dom(A) \cap \operatorname{int}\dom(B)\neq \varnothing$, hence $A+B$ is maximally monotone by \cite[Corollary 25.5(ii)]{bauschke11}.}  We note that $\bm \omega^{(k)} \in \dom(A)$, for all $k \in \bb N$. We also show in Section \ref{sec:NE_mon} that  for a point  ${\bm{\omega}}^{\star} = (\bm x^{\star}, \bm y^{\star}, \bm \mu^{\star}, \lambda^{\star}) \in \zer(A+B)$ with $A$ and $B$ defined in \eqref{eq:op_split}, $(\bm x^{\star}, \bm y^{\star})$ is a variational MS-GNE of the game \eqref{eq:gmi_game1}. Since we 
						\begin{itemize}
							\item {assume $\zeta,\gamma_i \in (0,\frac{1}{2\ell_B})$, for all $i \in \mc I$;
								\item use Legendre function $\varphi$ in \eqref{eq:Breg_f_gmi}, which {satisfies Assumption \ref{as:phi} and}  is  $1$-strongly convex {on $\dom(\mc T_1) \cap \operatorname{int}\dom(\varphi)$ (Lemma \ref{le:str_cvx_Bregf})};}
							\item initialize ${\bm{\omega}}^{(0)}={\bm{\omega}}^{(-1)} \in \operatorname{int}\dom (\varphi)$, and;
							\item assume that $\zer(A+B) \cap (\operatorname{int}\dom (\varphi))$ is non-empty,
						\end{itemize}
						the convergence of the sequence $({\bm{\omega}}^{(k)} )_{k \in \bb N}$ to a point in $\zer(A+B) \cap (\operatorname{int}\dom (\varphi))$ follows directly from Theorem \ref{th:conv_BreFoRB}. Hence, $(\bm x_k, \bm y_k)_{k \in \bb N}$ converges to a variational MS-GNE of the game in \eqref{eq:gmi_game1}. \qedd
						
						\begin{remark}
							{When the regularization function is in the form of Example \ref{ex:E_entropy}, as in \eqref{eq:y_upd}, \eqref{eq:mu_upd}, and \eqref{eq:lambda_upd}, we obtain Euclidean projection steps as in \eqref{eq:y_upd_alg_BFoRB}, \eqref{eq:mu_upd_alg_BFoRB}, and \eqref{eq:lambda_upd_alg_BFoRB}.	On the other hand, by choosing the negative entropy in Example \ref{ex:NE_entropy} as the regularization function when the constraint set is a simplex the optimizer on the right-hand side of \eqref{eq:x_upd} is of the closed form given by \eqref{eq:x_upd_alg_BFoRB}.   Another {pair of} regularization function and convex set that yields a closed-form non-Euclidean projection is given by the Fermi-Dirac entropy function on the box $[0,1]$, i.e., $\varphi^{\mathrm{FD}} \colon [0,1] \to \bb R \colon x \mapsto x \ln (x) + (1-x) \ln(1-x)$, which results in
								\begin{align*}
									y = \underset{x \in [0,1]}{\argmin} \left(\langle v, x \rangle + \dist_{\varphi_i^{\mathrm{FD}}}(x, z) \right) 
									= \frac{e^v z}{1-z + e^v z},
								\end{align*}
								for any $(v, z) \in \bb R \times (0,1)$.  
								Section V in \cite{gao20b} discusses other pairs. Although the authors of \cite{gao20b} consider a slightly different mirror operator, one can modify the closed-form mirror maps in \cite[Table 1]{gao20} to fit with the definition of our backward operator. \eod}  
						\end{remark}
						\section{Proof of Lemma \ref{le:cons_op}}
						\label{ap:pf:le:cons_op}
						The operator $\tilde{\mc T}_5$ in \eqref{eq:kkt_op2} 
						is a skew-symmetric linear operator, thus it is maximally monotone \cite[Example 20.35]{bauschke11}. Furthermore, due to Assumption \ref{as:conG}, $\|\L \| \leq \ell_{\tilde{\mc T}_5}= 2 \max_{i \in \mc I} \sum_{j=1}^N w_{i,j}$ \cite[Section 2.2]{yi19}. Then, it holds that $\|\tilde{\mc T}_5(\tilde{\bm{\omega}}_1)- \tilde{\mc T}_5(\tilde{\bm{\omega}}_2) \| \leq \|\L \| \|\tilde{\bm{\omega}}_1-\tilde{\bm{\omega}}_2 \| \leq \ell_{\tilde{\mc T}_5} \|\tilde{\bm{\omega}}_1-\tilde{\bm{\omega}}_2 \|$, for any $\tilde{\bm{\omega}}_1,\tilde{\bm{\omega}}_2$. 
						Thus, $\tilde{\mc T}_5$ 
						is Lipschitz continuous. \qedd

						\bibliographystyle{siamplain}
						\bibliography{ref}
\end{document}